\documentclass[preprint,
number, sort&compress,]{elsarticle}

\usepackage{graphicx}
\usepackage{amsmath}
\usepackage{amsfonts}
\usepackage{amssymb,latexsym}
\usepackage{fixmath}
\usepackage{mathrsfs,amsbsy}
\usepackage{enumerate}
\usepackage{algorithm,algorithmic}
\usepackage{multirow}
\usepackage{url}
\usepackage{float}
\usepackage{subfigure}
\usepackage{bm}
\usepackage{caption}
\usepackage{makecell}
\usepackage{xr}
\usepackage[]{hyperref}
\def\sss{\scriptscriptstyle}
\def\l{{\lambda}}
\def\scrH{\mathscr{H}}
\def\GO{\mathbb{O}^{\bGamma}}
\def\GQR{$\mathrm{\Gamma}${\rm QR}}
\def\GLan{$\mathrm{\Gamma}$-Lanczos}
\def\qD{\mbox{{\rm q}$\mathbb{D}$}}
\def\qU{\mbox{{\rm q}$\mathbb{U}$}}

\def\ba{\boldsymbol{a}}
\def\bb{\boldsymbol{b}}
\def\be{\boldsymbol{e}}
\def\bg{\boldsymbol{g}}
\def\bp{\boldsymbol{p}}
\def\bq{\boldsymbol{q}}
\def\br{\boldsymbol{r}}
\def\bu{\boldsymbol{u}}
\def\bv{\boldsymbol{v}}
\def\bx{\boldsymbol{x}}
\def\by{\boldsymbol{y}}
\def\bz{\boldsymbol{z}}
\def\bGamma{\boldsymbol{\Gamma}}
\def\bPi{\boldsymbol{\Pi}}
\def\b0{\bf{0}}

\def\wtd{\widetilde}

\DeclareMathOperator{\rank}{rank}
\DeclareMathOperator{\diag}{diag}

\DeclareMathOperator{\Span}{span}
\DeclareMathOperator{\sign}{sign}

\newtheorem{theorem}{Theorem}[section]
\newtheorem{lemma}[theorem]{Lemma}

\newtheorem{proposition}[theorem]{Proposition}
\newdefinition{definition}[theorem]{Definition}
\newdefinition{notation}{Notation}
\newdefinition{remark}{Remark}
\newdefinition{example}[theorem]{Example}
\newproof{proof}{Proof}
\usepackage{etoolbox}
\AtEndEnvironment{proof}{\qed}

\allowdisplaybreaks

\journal{Journal of Computational and Applied Mathematics}

\begin{document}
\raggedbottom
\begin{frontmatter}

	\title{Structure-Preserving \GQR\ and \GLan\
	Algorithms for Bethe-Salpeter Eigenvalue Problems}
	\author[Guo]{Zhen-Chen Guo\corref{cor1}} \ead{guozhch06@gmail.com}
	\author[seu]{Tiexiang Li} \ead{txli@seu.edu.cn}
	\author[seu]{Ying-Ying Zhou} \ead{zyy\_seu@126.com}
	\cortext[cor1]{Corresponding author}
	\address[Guo]{Department of Applied Mathematics, National Chiao Tung University, Hsinchu 300, Taiwan, R.O.C.} 
	\address[seu]{School of Mathematics, Southeast University, Nanjing, 211189, People's Republic of China}

\begin{abstract}
To solve the Bethe-Salpeter  eigenvalue  problem with distinct sizes,
two  efficient methods, called   \GQR \ algorithm and
\GLan\ algorithm,  are proposed in this paper. Both  algorithms  preserve
the special structure of the initial matrix $\scrH=\begin{bmatrix}A&B\\-\overline{B}&-\overline{A}\end{bmatrix}$,
resulting the computed eigenvalues and the associated eigenvectors still hold the  properties similar to  those of  $\scrH$.
Theorems are given to demonstrate the validity  of the proposed two algorithms  in theory.
Numerical results are presented to illustrate the  superiorities of our methods.
\end{abstract}

\begin{keyword}
Bethe-Salpeter  eigenvalue problem \sep  
$\mathrm{\Gamma}$-unitarity \sep 
\GQR\ algorithm \sep \GLan\ algorithm 

\end{keyword}

\end{frontmatter}

\section{Introduction}

In this paper, we consider the following  structured eigenvalue problem
\begin{equation}  \label{bsevp}
{\scrH\bx}\equiv
\begin{bmatrix}
\hphantom{-} A & \hphantom{-} B \\
-\overline{B}  & -\overline{A}%
\end{bmatrix}\begin{bmatrix}\bx_1\\ \bx_2\end{bmatrix}
=\lambda\, \bx,
\end{equation}%
where $A, B \in\mathbb{C}^{n\times n}$ with
$A^H=A,\ B^T=B$. Here, we denote by $A^H$, $\overline{A}$ and $B^T$, respectively, the conjugate transpose of $A$, the complex conjugate of $A$, and the transpose of $B$.
Such an eigenvalue problem \eqref{bsevp} is referred to as
a {\em Bethe-Salpeter eigenvalue problem} (BSEP). Any $\lambda\in\mathbb{C}$ and nonzero $\bx\in\mathbb{C}^{2n}$ that
satisfy \eqref{bsevp} are called an {\em eigenvalue\/} and its  corresponding 
{\em eigenvector} of $\scrH$, respectively. Accordingly, $(\lambda,\bx)$ is called an {\em eigenpair}.

The BSEP \eqref{bsevp}  arises from the discretization of 
the Bethe-Salpeter equation effectively  describing 
the bound states of a two-body quantum field theoretical system 
for the response function. Essentially, the Bethe-Salpeter equation 
comes from  the strict derivation 
of the problem within Green's-function theory, 
which can be used in 
the simulation of  electron-hole interaction effects (see 
\cite{dresselhausDSJ2007exciton}),  electron-positron 
interaction effects (see \cite{barbieriR1978solving}), 
etc.. More details about the Bethe-Salpeter 
equation can be found in     
\cite{onidaRR2002electronic, casida2009timedependent, parrishHSSM2013exact, reboliniTS2013electronic, 
sacs:10, rolg:10, casi:95} and the references therein.     
Usually, to solve the Bethe-Salpeter equation for a finite system all equations are 
projected onto an orthonormal spin-orbit basis, and then  we   obtain 
its corresponding discrete form \eqref{bsevp}.

Recently,  the BSEP \eqref{bsevp} is penetratively
studied under the condition that $\Gamma_{0}\scrH$ is positive
definite with $\Gamma_{0}=\diag(I_{n},  -I_{n})$ (see  \cite{yang:16,peter:15,peter:17}).
In \cite{yang:16}, the authors firstly show that  the  BSEP is equivalent to
a real Hamiltonian eigenvalue problem, then an efficient parallel algorithm is proposed
to compute all eigenpairs corresponding to those  positive eigenvalues of 
the equivalent real Hamiltonian matrix. 
By projecting the initial matrix $\scrH$ in \eqref{bsevp} onto a reduced basis set,  
\cite{peter:15, peter:17} apply low-rank and QTT tensor approximation to approximate the 
 BSEP  in large-scale. After that  a simplified  Bethe-Salpeter eigenvalue problem,  
with  the diagonal plus low-rank structure,  is solved.  
Though the methods proposed in \cite{yang:16,peter:15,peter:17} aim to solve the BSEP,
they all  are principally suitable for the  {\em linear response eigenvalue problem} (LREP),
a special case of the BSEP. Intrinsically, the LREP requires both
$A $ and $B$ 
are real symmetric matrices, i.e., 
$A^T=A\in \mathbb{R}^{n\times n}, B^T=B\in \mathbb{R}^{n\times n}$.   
A plenty of remarkable contributions have been made to solve  the LREP
and  please refer   to
 \cite{bali:12a, bali:13, teli:13, lllin:17} and references therein for more information.

People usually use  the classical QR algorithm and Lanczos  method to   
 compute all eigenvalues of a general matrix and 
a few number of eigenvalues and their associated eigenvectors 
of a Hermitian matrix, respectively.  
One can refer to \cite{demm:97, govl:96, saad2011numerical} and 
the references therein to find the details of  
the  QR algorithm and  the Lanczos  method. 
For the convergence of the Lanczos method please consult 
\cite{kaniel1966estimates, paige1971computation:thesis, saad1980rates}. 
The classical QR algorithm can be employed to solve the BSEP in modest  
size, however, the special structure of $\scrH$ (see the preliminary section) will not keep any longer. 
Also, we use    $\Pi$-tridiagonal matrices to solve the BSEP in this article, 
similarly to the Lanczos method. 
Many  literatures apply the tridiagonal or quasi-diagonal matrices 
to solve various problems, such as 
boundary value Poisson equations \cite{swarztrauber1974direct, 
smith1978numerical, sweet1973direct},  
circulant systems \cite{swarztrauber1977methods},   
time series analysis \cite{luatiP2010spectral}, 
control theory \cite{respondek2010numerical} and so on.  

In this paper, we will study a general BSEP  and the positive definite  restriction on 
$\Gamma_{0}\scrH$  is no longer assumed.
Upon the structure-preserving \GQR\ algorithm  proposed in \cite{lllin:17}, which
solves  the  LREP  by introducing the $\mathrm{\Gamma}$-orthogonality and performing
a series of $\mathrm{\Gamma}$-orthogonal transformations to the initial $\scrH\in \mathbb{R}^{2n\times 2n}$, 
we  put forward  a similar \GQR\ algorithm to solve  the BSEP   with
$\scrH\in \mathbb{C}^{2n\times 2n}$ being some dense  modest size matrix.  
Here,  $\mathrm{\Gamma}$-unitary transformations
are  used and  some implicit multishift  techniques are  employed.
As pointed in  \cite{yang:16},  the size of the BSEP  usually can be fairly large,
which actually is  proportional to the number of the electrons of one system.
To solve a  large-scale  BSEP,
a \GLan\   algorithm will be developed, where a decomposition
analogous to the classical Lanczos decomposition is constructed.
Some error bounds and a convergence theorem are given to illustrate 
the performance  of the proposed \GLan\   algorithm.

The rest of this paper is organized as follows. In Section~\ref{sec:defns}, we introduce
some basic definitions and give some preliminaries which will be used
in subsequent sections. The  \GQR\  algorithm to compute the BSEP  with modest  size 
is displayed in Section~\ref{sec:GG-QR}.
The  \GLan\   algorithm to solve the  large-scale BSEP
is developed in Section~\ref{sec:Lanalg}.
Numerical results  are presented in Section~\ref{sec:egs}.
Some concluding remarks are finally drawn in Section~\ref{sec:concl}.

\paragraph{\bf {Notations}} 
Throughout this paper,  we denote the  $j$-th column of the identity matrix $I$  by
 $\be_{j}$, whose size is determined  from  the context.
 The MATLAB expression, which specifies the submatrix with
the colon notation, will be used when necessary,
that is, $M(k:l, s:t)$ refers to the submatrix of 
$M$ formed by rows $k$ to $l$ and columns $s$ to $t$. 
For matrices $M$ and $N$, $\diag(M,N)$ represents 
the matrix  $\begin{bmatrix}
	M& \\ & N
\end{bmatrix}$. 
Additional, denote   $\Pi_{2k}=\begin{bmatrix}\bf 0&I_k\\I_k&\bf 0\end{bmatrix}$ for all positive integer $k$, and
$\Pi \equiv \Pi_{2n}$ for short, where $2n$ is the size of the matrix $\scrH$.

\section{Definitions and Preliminaries}\label{sec:defns}
Some definitions and preliminaries are  presented in this section,
where many  definitions and propositions  are quoted from  \cite{lllin:17},
most of which can be trivially extended
to the complex matrix $\scrH$.

\begin{definition}\label{df2.1}
Let $G\in \mathbb{C}^{2n\times 2m}$, where $m\leq n$.
$G$ is a {\em $\Pi^{\pm}$-matrix} if $G\Pi_{2m}=\pm\Pi_{2n}\overline{G}$,  
that is, $G$ is of  the form
\begin{equation}  \label{eq5}
G=
\begin{bmatrix}
\hphantom{-} G_1 & \hphantom{-} G_2 \\
\pm \overline{G}_2 & \pm \overline{G}_1%
\end{bmatrix}
\quad \text{ with }\quad G_1, \ G_2 \in\mathbb{C}^{n\times m}.
\end{equation}
Denote by $\bPi_{2n\times 2m}^{\pm}$
the set of all $2n\times 2m$  {\em $\Pi^{\pm}$-matrices}. When $m=n$, we simply write
$\bPi_{2n}^{\pm}$ for $\bPi_{2n\times 2n}^{\pm}$.
\end{definition}

\begin{definition}
\label{dfprin}
Let $G\in \bPi_{2n}^{\pm}$ as in \eqref{eq5} and set
\[
G_{1i}=G_1(1:i,1:i), \quad
G_{2i}=G_2(1:i,1:i).
\]
$\begin{bmatrix}
G_{1i} & G_{2i} \\
\pm \overline{G}_{2i} & \pm \overline{G}_{1i}%
\end{bmatrix}$ is called the $i$-th {\em $\Pi^{\pm }$-leading principal submatrix} of $G$
and its determinant is called the $i$-th {\em $\Pi^{\pm }$-leading
principal minor} of $G$.
\end{definition}

\begin{definition}\label{df2.2}
Let $G\in\bPi_{2n\times 2m}^{\pm}$ as in \eqref{eq5}.
\begin{enumerate}
  \item $G$ is  {\em $\Pi^{\pm}$-upper} ({\em $\Pi^{\pm}$-quasi-upper}) {\em triangular}, if
        $G_1$ is upper (quasi-upper)  triangular and $G_2$ is strictly upper triangular.

        Denote by $\mathbb{U}^{\pm}_{2n\times 2m}$ ($\qU^{\pm}_{2n\times 2m}$) the set  
		of all $2n\times 2m$  {\em $\Pi^{\pm}$-upper} ({\em $\Pi^{\pm}$ quasi-upper}) 
		{\em triangular matrices}, and write, for short,
        $\mathbb{U}^{\pm}_{2n}=\mathbb{U}^{\pm}_{2n\times 2n}$ and $\qU^{\pm}_{2n}=\qU^{\pm}_{2n\times 2n}$.
  \item $G$  is  {\em $\Pi^{\pm }$-diagonal} ({\em $\Pi^{\pm }$-quasi-diagonal}), if $G_{1}$
        is diagonal (quasi-diagonal) and $G_{2}$ is diagonal.

         Denote by $\mathbb{D}_{2n\times 2m}^{\pm }$
         ($\qD^{\pm}_{2n\times 2m}$) the set of  all $2n\times 2m$ 
		 {\em $\Pi^{\pm }$-diagonal} ({\em $\Pi^{\pm }$-quasi-diagonal}) {\em matrices},
         and write, for short,
         $\mathbb{D}^{\pm}_{2n}=\mathbb{D}^{\pm}_{2n\times 2n}$ and $\qD^{\pm}_{2n}=\qD^{\pm}_{2n\times 2n}$.
  \end{enumerate}
\end{definition}

Let $\mathbb{J}_{n}$ be the set of all $n\times n$ diagonal matrices 
with $\pm 1$ on the main diagonal and set
\begin{equation*}
\bGamma_{2n}=\{\diag(J,-J)\,:\,J\in \mathbb{J}_{n}\}.
\end{equation*}%
\label{eqw2.1} Note that $\Gamma_{0}=\diag(I_{n},-I_{n})\in \bGamma%
_{2n}$.

\begin{definition}\label{df2.3}
\begin{enumerate}
	\item Let $G\in\bPi^{-}_{2n}$ as in \eqref{eq5}. 
		$G$ is {\em $\Pi^{-}$-Hermitian} ({\em $\Pi^{-}$-Hermitian-tridiagonal)} 
		with respect to $\Gamma=\diag(J, -J)\in \bGamma_{2n}$ if $J G_1$ is Hermitian 
		(Hermitian-tridiagonal) and $J G_2$ is symmetric (diagonal).
	\item Let $G\in\bPi_{2n}^{+}$ as in \eqref{eq5}. $G$ is {\em $\Pi^{+}$-Hermitian} 
		({\em $\Pi^{+}$-Hermitian-tridiagonal}) with respect to 
		$\Gamma=\diag(J,-J)\in \bGamma_{2n}$   
		if $J G_1$ is  Hermitian (Hermitian-tridiagonal) and 
		$J G_2$ is skew-symmetric (zero matrix).  
  \end{enumerate}
\end{definition}

\begin{remark}\label{remark1}
\begin{enumerate}
	\item Both {\em $\Pi^{-}$-Hermitian} and {\em $\Pi^{+}$-Hermitian} matrices $G$, with 
		respect to $\Gamma\in\bGamma_{2n}$, satisfy $(\Gamma G)^H = \Gamma G$.  
	\item $\scrH$ is  {\em $\Pi^{-}$-Hermitian} with respect to $\Gamma_0$. And both 
			$\scrH^2$ and $\scrH^4$ are {\em $\Pi^{+}$-Hermitian} with respect to $\Gamma_0$. 
\end{enumerate}
\end{remark}

The following result   characterizes the proposition of 
the eigenvalues of the matrix $G\in\bPi_{2n}^-$, 
and then  $\scrH$. 

\begin{proposition}
\label{prop2.2}
Let $G\in\bPi_{2n}^-$. Then $G$ has $2n$ eigenvalues
appearing in pairs $(\lambda, -\overline{\lambda})$,
which degenerates to $(\lambda,-\lambda)$ for real  $\lambda$.
\end{proposition}
\begin{proof}
It holds that
\[
	\det(G-\l I)=(-1)^n \det ( G\Pi-\l \Pi)=(-1)^n \det(-\Pi \overline{G} -\l \Pi)= \overline{\det(G+\overline{\l} I)},
\]
then  the assertion follows immediately.
\end{proof}

\begin{definition}
\label{df2.4}
$Q\in\bPi_{2n\times 2m}^+$ is {\em $\mathrm{\Gamma}$-orthonormal\/}
with respect to $\Gamma\in\bGamma_{2n}$ if $\Gamma':=Q^H \Gamma Q\in \bGamma_{2m}$;
if also $m=n$, we say
it is {\em $\mathrm{\Gamma}$-unitary} with respect to $\Gamma$.
\end{definition}

Denote by $\GO_{2n\times 2m}$  the set of all $2n\times 2m$ 
{\em $\mathrm{\Gamma}$-orthnormal matrices}
and write $\GO_{2n}=\GO_{2n\times 2n}$ for simplicity.
Often, for short, we say $Q\in\GO_{2n\times 2m}$ is {\em $\mathrm{\Gamma}$-orthnormal} 
({\em $\mathrm{\Gamma}$-unitary}), which implies that there exists some $\Gamma\in\bGamma_{2n}$ such that $Q^H \Gamma Q\in \bGamma_{2m}(\bGamma_{2n})$.
More propositions on   matrices  in $\GO_{2n}$ please refer to \cite{lllin:17}.

The following lemmas further illustrate the eigen-structure of  $\scrH$. 

\begin{lemma}\label{lemma-H-hermitian}
It holds that $\Gamma_0\scrH \Gamma_0=\scrH^{H}$.
\end{lemma}
\begin{proof}
It can be verified directly.
\end{proof}

\begin{lemma}\label{lemma-eigenpairs}
Let $\left( \lambda ,\bx\right)$ be an eigenpair of $\scrH$,
then $\left( -\overline{\lambda} ,\Pi\overline{ \bx}\right) $ is
also an eigenpair of $\scrH$. Furthermore, if $\mbox{Im}(\lambda)\neq0$, then
$\left( -\lambda,(\Gamma_0\Pi \bx)^T \right) $ and $\left( \overline\lambda,(\Gamma_0 \bx)^H\right)$
act as the left eigenpairs of $\scrH$. 
\end{lemma}
\begin{proof}
The results follow from Lemma \ref{lemma-H-hermitian} and the fact
$\scrH\Pi =-\Pi {\overline{\scrH}}$.
\end{proof}

\begin{lemma}
Let $(\lambda,\bx)$ and $(\mu,\by)$ be two  eigenpairs  of $\scrH$ with $\mu\neq \overline{\lambda}$,
it then holds that $x^{H}\Gamma_0 y=0$.
\end{lemma}

\begin{proof}
It follows from Lemma \ref{lemma-H-hermitian} that $(\Gamma_0x)^{H}\scrH=\overline{\lambda}(\Gamma_0x)^{H}$,
leading to $\mu x^{H}\Gamma_0y=x^{H}\Gamma_0\scrH y=\overline{\lambda}(\Gamma_0x)^{H}y$.
Then the result holds since $\mu\neq \overline{\lambda}$.
\end{proof}

\section{$\mathrm\Gamma$QR  Factorization and $\mathrm\Gamma$QR  Algorithm}\label{sec:GG-QR}
In this section, the \GQR\ factorization  developed in \cite{lllin:17} will be extended to the complex $\Pi^\pm$-matrix.
Then a \GQR\ algorithm similar to that  in  \cite{lllin:17} is proposed to solve  the BSEP
\eqref{bsevp} with  modest  size.

\subsection{\GQR\ Factorization} \label{sec:GG-QRFC}
\begin{definition}\label{df3.1}
$G=QR$ is called a
{\em \GQR\ factorization} of $G\in\bPi_{2n\times 2m}^{\pm}$ with respect to $\Gamma\in\bGamma_{2n}$ if
$R\in\mathbb{U}_{2n\times 2m}^{\pm }$ and
$Q\in\GO_{2n}$ with respect to $\Gamma$
or if $R\in\mathbb{U}_{2m}^{\pm }$ and
$Q\in\GO_{2n\times 2m}$ with respect to $\Gamma$.
\end{definition}

We call the case  that  $R\in\mathbb{U}_{2m}^{\pm }$ and
$Q\in\GO_{2n\times 2m}$ as a {\em skinny\/} \GQR \  factorization.
Theorem \ref{thm:GQR-basic} reveals that  (\romannumeral1) for a  given
$\Gamma\in\bGamma_{2n}$,  almost every  $G\in\bPi_{2n\times 2m}^{\pm}$
has a \GQR\ factorization with respect to $\Gamma$;
 (\romannumeral2) the \GQR\  factorization is unique upon  diagonal transformations, 
 whose diagonal entries are with moduli  $1$.



\begin{theorem}\label{thm:GQR-basic}
Suppose that $G\in\bPi_{2n\times 2m}^{\pm}  (m\leq n) $ is of full column rank and 
$\Gamma\in\bGamma_{2n}$.
\begin{enumerate}
\item[{\em (\romannumeral1)}] $G$ has a \GQR\ factorization with respect to
        $\Gamma$ if and only if no $\Pi^{-}$-leading principal minor of $G^H\Gamma G$ 
		vanishes; and 
\item[{\em (\romannumeral2)}] let $G=QR=\widetilde Q \widetilde R$
        be two  skinny \GQR\ factorizations of $G$ with respect to $\Gamma$,
        that is, $\widetilde Q^H\Gamma\widetilde Q=Q^H\Gamma Q\in\bGamma_{2m}$, 
        then there is a diagonal matrix $\Theta=\diag(D, \overline{D})$ with
        $D=\diag(\delta_1, \ldots, \delta_m), |\delta_j|=1, j=1, \ldots, m$,
        such that $  Q=\widetilde Q \Theta$
        and $R=\Theta^H \widetilde R$. Particularly, when all diagonal entries
        in the top-left quarters of $R$ and $\wtd R$ are positive, we have  $\Theta=I_{2m}$,
        implying  $Q=\wtd Q$,   $R=\wtd R$.
%
\end{enumerate}
\end{theorem}

\begin{proof}
The proof can be parallel  drawn from  \cite{lllin:17}.
\end{proof}

We can obtain  a \GQR\ factorization for a  given $G\in\bPi_{2n\times 2m}^{\pm}$  by
performing  a sequence of  hyperbolic  Householder transformations   and some  hyperbolic Givens transformations.
Before giving both transformations, we  present  the
Householder-like  transformation introduced in \cite{buns:81},
which is $(J_1, J_2)$-unitary 
and can   simultaneously eliminate some elements of a given
vector $\bu\in\mathbb{C}^k$.

Let $\ba=\begin{bmatrix}
	\alpha_1 & \dots & \alpha_k
\end{bmatrix}
\in\mathbb{C}^k (1\leq k \leq n)$, $J=\diag(\jmath_1, \cdots, \jmath_k)\in\mathbb{J}_k$,
which satisfy that $\ba^H J\ba\ne0$.
Choose $r$ $(2\leq r\leq k)$ so that $\jmath_{r}\ba^{H}J\ba>0$. Now  set $P_u\in\mathbb{C}^{k\times k}$ be the
permutation matrix which  interchanges   the first row  and the $r$-th row.
Write  $\hat{\ba}=P_u\ba=\begin{bmatrix}\hat{\alpha}_1&\cdots&\hat{\alpha}_k\end{bmatrix}^T $
	and $\widehat{J}=P_uJP_u=\diag(\hat{\jmath}_1, \ldots, \hat{\jmath}_k)$, where
$\hat{\jmath}_1 \ba^H J \ba=\hat{\jmath}_1 \hat{\ba}^H \widehat{J} \hat{\ba}>0$.
A Householder-like transformation proposed in \cite{buns:81} is to eliminate
all elements $\alpha_2,\dots,\alpha_k$ as follows. Let
\begin{align}\label{Householder}
\widehat H(\ba)^{-1}=I-\frac{\hat{\jmath}_1}{\beta}(\hat{\ba}-\alpha \be_1)(\hat{\ba}-\alpha \be_1)^{H}\widehat{J}, \qquad
H(\ba)^{-1}= \widehat H(\ba)^{-1}P_u,
\end{align}
where $\alpha=-\sign(\hat{\alpha}_1)\sqrt{\hat{\jmath}_1 \hat{\ba}^H \widehat{J} \hat{\ba}}
=-\sign(\hat{\alpha}_1)\sqrt{\hat{\jmath}_1 \ba^H J \ba} $ and $\beta=\overline{\alpha}(\alpha-\hat{\alpha}_1)$.
It then holds that
\begin{equation}\label{orth-householder}
H(\ba)^{-1}\ba = [\widehat H(\ba)^{-1}P_u]\ba=\widehat H(\ba)^{-1}\hat{\ba}=\alpha \be_1
,\ \
(H(\ba))^H J {H}(\ba)=\widehat{J}.
\end{equation}

\paragraph{Hyperbolic Householder transformation}
Let $\bu\in\mathbb{C}^{2n}$,
\[
	\Gamma=\diag(\gamma_1,\cdots,\gamma_n,-\gamma_1,\cdots,-\gamma_n)\in\bGamma_{2n},
\]
and $1\leq \ell <m\leq n$.
We are to zero out either the $(l+1)$-th to the $m$-th
elements or the $(n+l+1)$-th to the $(n+m)$-th elements of $\bu$ by applying the Householder-like transformation given above.
There are two cases:
\[
	\begin{cases}
\text{Case 1:}&  \ba\leftarrow\bu_{(\ell:m)},\ J=\diag(\gamma_\ell, \ldots, \gamma_{m});  \\
\text{Case 2:}& \ba\leftarrow\bu_{(\ell':m')} \text{ with } \ell'=\ell+n \text{ and } m'=m+n,  \ J=\diag(\gamma_\ell, \ldots, \gamma_{m}).
	\end{cases}
\]

Using \eqref{orth-householder} we can construct a hyperbolic Householder
 transformation $Q$
as follows:
\[
Q^{-1}=\diag(I_{l-1}, \  H^{-1}(\bu), \ I_{n-m}, \ I_{l-1}, \ \overline{H}^{-1}(\bu), \ I_{n-m}).
\]
Clearly, it  holds that
\[
Q^{H}\Gamma_{2n}Q=\widetilde{\Gamma}_{2n}
=\diag(\tilde{\gamma}_1,\cdots,\tilde{\gamma}_n,-\tilde{\gamma}_1,\cdots,-\tilde{\gamma}_n)\in\bGamma_{2n}
\]
with $\tilde{\gamma}_l=\gamma_r, \tilde{\gamma}_r=\gamma_l$ and
$\tilde{\gamma}_j=\gamma_j$ for $j\neq l, \ r$, 
where $r$ determines the permutation matrix $P_u$, i.e., the $r$-th row  permute with the first row of $\bu$.

The following introduced hyperbolic Givens transformation is used to zero out one element of a $2n$-length vector.

\paragraph{Hyperbolic Givens transformation}
Let $\begin{bmatrix}\ba^T  & \bb^T  \end{bmatrix}^T $ with $\ba, \bb\in\mathbb{C}^{n}$
and $\Gamma_{2n}=\diag(\gamma_1,\cdots,\gamma_n,-\gamma_1,\cdots,-\gamma_n)\in\bGamma_{2n}$.
 We are to eliminate the $l$-th entry of $\bb$ with the $l$-th one in $\ba$.
Denote by $\alpha=\ba(l)$, $\beta=\bb(l)$ and define
\[
	\begin{cases}
&c=\dfrac{1}{\sqrt{1-r^{2}}},\ s=\dfrac{r}{\sqrt{1-r^{2}}}\ \ \text{ with } \
                   r=\dfrac {\beta}{\alpha}, \quad  \text{ if }\ |\alpha |>|\beta |; \\
&c=\dfrac{r}{\sqrt{1-r^{2}}},\ s=\dfrac{1}{\sqrt{1-r^{2}}}\ \ \text{ with } \
                   r=\dfrac {\alpha}{\beta},  \quad \text{ if }\ |\alpha |<|\beta |.%
	\end{cases}
\]
Now we define   the hyperbolic Givens transformation $Q$  through its inverse:
\[
Q^{-1}=\begin{bmatrix}C&S\\ \overline{S}&\overline{C}\end{bmatrix} \, \text{with} \,
C=\diag(1, \ldots, c, 1, \ldots, 1), \ S=\diag(0, \ldots, 0, s, 0, \ldots, 0),
\]
where $c$ and $s$, respectively,  are  the $l$-th diagonal elements of $C$ and $S$.
It is simple to verify that
\[
	Q^{-1}\Gamma_{2n}Q^{-H}=\widetilde{\Gamma}_{2n}
=\diag(\tilde{\gamma}_1, \ldots, \tilde{\gamma}_n, -\tilde{\gamma}_1, \ldots, -\tilde{\gamma}_n),
\]
where $\tilde{\gamma}_j=\gamma_j$ when $j\neq l$ $(1\leq j\leq n)$ and $\tilde{\gamma}_l=\gamma_l(c^{2}-s^{2})$,
suggesting that $Q^{H}\Gamma_{2n} Q \equiv \widetilde{\Gamma}_{2n}\in\bGamma_{2n}$.

\begin{remark}
The hyperbolic rotation parameters $c$ and $s$ 
will not exist if $|\alpha|=|\beta|$.
However, as claimed in \cite{fllw:92},
such a case may occur  when the matrix is artificially designed.
Clearly, it  would be  numerical instability
once $|\alpha|$ is pathologically close to $|\beta|$, where serious cancellations  could
occur as discussed  in \cite{buns:81}.
Nevertheless, it is possible to reorganize the  computation process  to successfully
avoid the cancellations (see \cite{alpp:88}).
\end{remark}

\subsection{\GQR\ Algorithm}\label{sec:GQRalg}
Although the  discussion in the last subsection is applicable to all $\Pi^\pm$-matrices,
we only  focus on the  $\Pi^{-}$-Hermitian matrix 
$\scrH$ with respect to $\Gamma_0 \in\bGamma_{2n}$ in this subsection. 
We  will extend  the structure-preserving \GQR\ algorithm in   
\cite{lllin:17}  to the BSEP \eqref{bsevp}.

Generally, we shall  compute   a sequence of $\mathrm{\Gamma}$-unitary matrices $\{Q_{i}\}$,
based on the  \GQR\ factorizations of $\scrH_j$, such that
\begin{align}\label{eq-iteration}
\scrH_{j+1}=Q^{-1}_j \scrH_j Q_j,\quad
Q^H_j \Gamma_jQ_j=\Gamma_{j+1}
\quad\mbox{for $j=1,2,\ldots $},
\end{align}
where initially $\scrH_0 \equiv \scrH$. In practical, for the sake of  many numerical concerns, including
structure preserving, numerical stability and the amount of calculations,
as in the  classical QR algorithm,  we  firstly reduce $\scrH$ to a $\Pi^-$-Hermitian-tridiagonal
matrix with respect to some $\Gamma\in \bGamma_{2n}$, by applying a series
of hyperbolic Householder transformations and hyperbolic Givens transformations.
Then  a  implicit  \GQR\ factorization will be proceed with,
where the shift technique is incorporated  in  the whole process to
accelerate the convergence.

Here, we just give   a condensed description of the   implicit multishift   \GQR\ algorithm.
Readers who are interested in details
please  refer to \cite{lllin:17}. Firstly, we  reduce $\scrH$ to a
$\Pi^-$-Hermitian-tridiagonal matrix $\scrH_1$, and then
a $\mathrm{\Gamma}$-unitary transformation $Q$ is constructed to
rotate the first column of $p(\scrH_1)$ to a vector parallel to
$\be_1$, where the  filtering polynomial $p(x)$ is defined as
\begin{equation}  \label{px}
\begin{cases}
p(x)=(x-{\lambda})(x+{\lambda}) & \text{for real or purely imaginary ${\lambda}$};
\\
p(x)=(x-{\lambda})(x+{\lambda})(x-\overline{\lambda})(x+\overline{\lambda}) & \text{%
for complex ${\lambda}$}.
\end{cases}%
\end{equation}
Note that  $p(\scrH_1)\in\bPi_{2n}^+$. Ultimately, some
hyperbolic Householder transformations and hyperbolic Givens transformations, which are
$\mathrm{\Gamma}$-unitary,  will be pursued to
fit $Q^{-1} \scrH_1 Q$ to a new $\Pi^-$-Hermitian-tridiagonal matrix.

The following part devotes to reveal the convergence of the  \GQR\ iteration  \eqref{eq-iteration} without
cooperating with any  shift strategies.  

\begin{lemma}\label{lemma-decomp}
Assume that all eigenvalues of $\scrH$ are simple, then there exist $X\in\bPi_{2n}^+$ and $\Lambda\in\mathbb{D}^{-}_{2n}$
 such that  $\scrH X=X\Lambda$. Furthermore, let
 $P=[\be_1, \be_{n+1}, \be_{2}, \be_{n+2}, \cdots, \be_{n}, \be_{2n}]\in\mathbb{R}^{2n\times 2n}$,
 it then holds that
\[
	P^T  \Lambda P=\diag(\Lambda_1,\cdots,\Lambda_\ell)
\]
 with    $\lambda(\Lambda_i)=\{\pm\lambda_i\}$ or
$\lambda(\Lambda_i)=\{\pm\lambda_i,\pm\overline{\lambda}_i\}$.
\end{lemma}

\begin{proof}
Obviously, we just need to prove that for the simple  eigenvalue $\imath\omega (0\neq\omega\in\mathbb{R})$,
there exists   $[\bu\ \bv]\in\bPi_{2n\times 2}^+$ such that
\begin{equation}\label{hxys}
\scrH[ \bu \ \bv]=[\bu \ \bv]\begin{bmatrix} 0&\omega \\ -\omega &0\end{bmatrix}.
\end{equation}
Let  $\scrH\bx=\imath\omega\bx$ with $\|\bx\|_2=1$, it then  holds that
$\scrH(\Pi\overline{\bx})=\imath\omega(\Pi\overline{\bx})$, yielding
 $\bx=\varrho^2(\Pi\overline{\bx})$ with $0\neq\varrho\in\mathbb{C}$.
 Now defining $\widehat{\bx}\equiv \bar{\varrho}\bx$ we have  $\scrH\widehat{\bx}=\imath\omega\widehat{\bx}$ with
 $\widehat\bx=\Pi\overline{\widehat \bx}$.
Similarly, there exists $\by=\Pi\overline{\by}\in\mathbb{C}^{2n}$ such that  $\scrH\by=-\imath\omega\by$.
By noting that
\begin{align*}
\begin{bmatrix} 0&\omega \\ -\omega &0\end{bmatrix}\begin{bmatrix} 1&\imath\\ \imath&1 \end{bmatrix}=\begin{bmatrix} 1&\imath\\ \imath&1 \end{bmatrix}\begin{bmatrix} \imath\omega &\\ & -\imath\omega \end{bmatrix} \mbox{ and }  \begin{bmatrix} 1&\imath\\ \imath&1 \end{bmatrix}^{-1}=\frac{1}{2}\begin{bmatrix} 1&-\imath\\- \imath&1 \end{bmatrix}\equiv S,
\end{align*}
it then follows that
\begin{equation}\label{hxy1}
\scrH [\bx\ \ \by]S=[\bx\ \ \by]S\left(S^{-1}\begin{bmatrix} \imath\omega &\\ & -\imath\omega \end{bmatrix}S\right)=[\bx\ \ \by]S\begin{bmatrix} 0&\omega \\ -\omega &0\end{bmatrix}.
\end{equation}
With defining  $[\bu,\bv]=2e^{\imath\frac{\pi}{4}}[\bx,\by]S$, the result  \eqref{hxys} holds.
\end{proof}

\begin{theorem}\label{conv}
Let $ \scrH X=X\Lambda$ be the decomposition of $\scrH$ specified in Lemma \ref{lemma-decomp},
and suppose that  $|\lambda_{1}|>\cdots >|\lambda_{\ell }|>0$. Provided that all \GQR\
factorizations of $\scrH_j$ in \eqref{eq-iteration}  exist,
then if the \GQR\  factorization of $X$ with respect to $\Gamma_0$
and the $\Pi^{+}$-LU factorization of $X^{-1}$  exist (see \cite{lllin:17} for the $\Pi^{+}$-LU factorization),
then the sequence $\{\scrH_{j}\}$ generated in
the \GQR~iteration \eqref{eq-iteration} will converge to a $\Pi^-$-quasi-diagonal matrix
with its eigenvalues emerging in the order of $\lambda_1, \lambda_2,\cdots, \lambda_{\ell}$,
as $j\rightarrow \infty $.
\end{theorem}

\begin{proof}
The proof  can be parallelly  drawn from  \cite{lllin:17}.
\end{proof}

\begin{remark}\label{remark3}
 Actually, the \GQR~iteration~\eqref{eq-iteration} can  be applied to 
 a $\Pi^+$-Hermitian matrix $\mathscr{W}$ (defined in \ref{df2.3}) 
 with respect to $\Gamma\in\bGamma_{2n}$. 
 Different from the $\Pi^{-}$-Hermitian matrix $\scrH$, when applying 
 the \GQR~algorithm to $\mathscr{W}$, it reduces $\mathscr{W}$ 
 to a $\Pi^+$-Hermitian-tridiagonal matrix firstly. 
 Then   a sequence of $\Pi^+$-Hermitian-tridiagonal 
 $\{\mathscr{W}_j\}$ are computed by the \GQR~iteration. 
 Similarly to Theorem~\ref{conv}, the generated sequence   $\{\mathscr{W}_j\}$ 
 will converge to a quasi-diagonal matrix $T\in\bPi_{2n}^{+}$.
\end{remark}

\section{\GLan\   Theory and \GLan\   Algorithm}\label{sec:Lanalg}

Inspired  by the classical Lanczos method (please refer to
\cite{saad2011numerical, demm:97}),
an honorable method for computing large-scale eigen-problems of Hermitian matrices,
in   this section we propose a  Lanczos-like
algorithm, named \GLan\   algorithm, to solve the eigen-problem of  the special matrix $\scrH$ when $n$ is fairly large.
Firstly, a special subspace derived from  the Krylov subspace  will be introduced, which
is named as the $\bPi$- Krylov subspace.

\subsection{\GLan\   decomposition}\label{subsection4.1}
\begin{definition}\label{krylov-space}
The $k$-th order $\bPi$-Krylov subspace of $\scrH$ with respect to
any arbitrary given vector $\bq_1\in\mathbb{C}^{2n}$ is defined as
\[
\mathcal{K}_{2k}(\scrH, \bq_1)=\Span\left\{\bq_1, \scrH \bq_1, \ldots,
\scrH^{k-1}\bq_1, \Pi\overline{\bq}_1,
\Pi\overline{\scrH}\overline{\bq}_1, \ldots, \Pi\overline{\scrH}^{k-1}\overline{\bq}_1\right\}.
\]
The corresponding matrix
\begin{align}\label{K_2k}
K_{2k}(\scrH, \bq_1)=
\begin{bmatrix}\bq_1&\scrH\bq_1&\cdots&\scrH^{k-1}\bq_1&
\Pi\overline{\bq}_1&\Pi\overline{\scrH}\overline{\bq}_1&\cdots&\Pi\overline{\scrH}^{k-1}\overline{\bq}_1\end{bmatrix}
\end{align}
is referred to as the  $k$-th order $\bPi$-Krylov matrix with respect to $\bq_1$.
\end{definition}

Clearly, the definition of $\bPi$-Krylov matrix reveals that $K_{2k}(\scrH, \bq_1)$ is a $\Pi^{+}$ matrix.
Additional, it apparently   holds that $\Pi\overline{\scrH}\overline{\bp}=
-\scrH \Pi\overline{\bp}$ for any arbitrary $\bp\in\mathbb{C}^{2n}$,
which leads to the following lemma.

\begin{lemma}\label{lemma-PiHl}
For any $\bp\in\mathbb{C}^{2n}$, we have $\Pi\overline{\scrH}^{l}\overline{\bp}=
(-1)^l\scrH^l\Pi\overline{\bp}$ for all $l\geq0$.
\end{lemma}

\begin{theorem}\label{theorem-K-2k-2k+2}
Suppose that all $\Pi^{-}$-leading principal minors of
 \[
 K_{\Gamma_0}:=K_{2k+2}^{H}(\scrH, \bq_1)\Gamma_0 K_{2k+2}(\scrH, \bq_1)
 \]
 are not vanish
for  the given vector $\bq_1\in\mathbb{C}^{2n}$.
Let
\[
K_{2k+2}(\scrH, \bq_1)=Q_{2k+2}R_{2k+2}
\]
be a  \GQR \ factorization with respect to $\Gamma_0$,
namely  $Q^{H}_{2k+2}\Gamma_{0}Q_{2k+2}=\Gamma_{2k+2}$
for some $\Gamma_{2k+2}\in\bGamma_{2k+2}$. Write
\begin{align}
&R_{2k+2}=\left[\begin{array}{c|c|c|c}
R_t& \br_t& \overline{R}_s &\overline{\br}_s\\
& &&\\[-2mm]
\hline
& &&\\[-2mm]
\mathbf&\tau_{k+1}&\mathbf&0\\
&&& \\[-2mm]
\hline
& &&\\[-2mm]
R_s&\br_s&\overline{R}_t&\overline{\br}_t\\
& &&\\[-2mm]
\hline
& &&\\[-2mm]
\mathbf&0&\mathbf&\overline{\tau}_{k+1}\\
\end{array}\right],  &
R_{2k}=\begin{bmatrix}R_t&\overline{R}_s\\R_s&\overline{R}_t\end{bmatrix},&  \label{R-2k+2} \\
&Q_{2k+2}=\begin{bmatrix}U_k&\bu_{k+1}&\Pi\overline{U}_k&\Pi\overline{\bu}_{k+1}\end{bmatrix}, &
Q_{2k}=\begin{bmatrix}U_k&\Pi\overline{U}_k\end{bmatrix}, & \label{Q_2k+2}\\
&\Gamma_{2k+2}=\diag(D_k, \delta_{k+1},  -D_k, -\delta_{k+1}), &
 \Gamma_{2k}=\diag(D_k, \ -D_k),\notag&
\end{align}
where $R_t\in\mathbb{C}^{k\times k}$ and $R_s\in\mathbb{C}^{k\times k}$ respectively
are  upper triangular and strictly  upper triangular, $\br_{t}, \br_{s}\in\mathbb{C}^k$, $\tau_{k+1}\in\mathbb{C}$,
$U_k\in\mathbb{C}^{2n\times k}$, $\bu_{k+1}\in\mathbb{C}^{2n}$,
$D_k\in\mathbb{C}^{k\times k}$, $\delta_{k+1}\in\mathbb{C}$,
and denote    the $(k,k)$ element of $R_{2k}^{-1}$ by  $\zeta_{kk}$,
it then  holds that
\begin{enumerate}
 \item[{\em (\romannumeral1)}] $R_{2k}$ is nonsingular and
                   \[
                   R_{2k}^{-1}=\begin{bmatrix}(R_t-\overline{R}_s\overline{R}_t^{-1}R_s)^{-1}&
                   -(R_t-\overline{R}_s\overline{R}_t^{-1}R_s)^{-1}\overline{R}_s\overline{R}_t^{-1}\\
                   -(\overline{R}_t-R_sR_t^{-1}\overline{R}_s)^{-1}R_sR_t^{-1}&
                   (\overline{R}_t-R_sR_t^{-1}\overline{R}_s)^{-1}\end{bmatrix},
                   \]
                   whose $(1,1)$ and $(1,2)$ blocks are upper triangular and strictly upper triangular, respectively; and
 \item[{\em (\romannumeral2)}] $K_{2k}(\scrH, \bq_1)=Q_{2k}R_{2k}$; and
 \item[{\em (\romannumeral3)}] $Q_{2k}^{H}\Gamma_{0}Q_{2k}=\Gamma_{2k}$; and
 \item[{\em (\romannumeral4)}] $\scrH Q_{2k}=Q_{2k+2}\widetilde{T}_{2k}$  and
                  $Q_{2k}^{H}\Gamma_{0}\scrH Q_{2k}=\Gamma_{2k}\begin{bmatrix}T_{11}& -\overline{T}_{21} \\
                  T_{21}& -\overline{T}_{11}\end{bmatrix}$,  where $\widetilde{T}_{2k}$
				  is of the  form
				  \begin{align}\label{T_2k}
					  \widetilde{T}_{2k}=\begin{bmatrix}T_{11}&-\overline{T}_{21}\\ \zeta_{kk}\tau_{k+1}\be_k^T &0\\
					  T_{21}&-\overline{T}_{11}\\0&-\overline{\zeta}_{kk}\overline{\tau}_{k+1}\be_k^T \end{bmatrix}
					  \in\mathbb{C}^{(2k+2)\times 2k}
				  \end{align}
                  with $T_{11}, T_{21}\in\mathbb{C}^{k\times k}$  respectively
				  being  unreduced  tridiagonal and  diagonal; and
			  \item[{\em (\romannumeral5)}] $(D_kT_{11})^{H}=D_kT_{11}$ with $T_{11}$ being the matrix
				  defined in {\em(\romannumeral4)}; and
 \item[{\em (\romannumeral6)}] $Q^{H}_{2k}\Gamma_0\bu_{k+1}=0$ and
                  $Q^{H}_{2k}\Gamma_0\Pi\overline{\bu}_{k+1}=0$.
\end{enumerate}
 \end{theorem}

\begin{proof}
	For the result in (\romannumeral1), let $P\in\mathbb{R}^{2k\times 2k}$ be the permutation matrix with
\[
P=\begin{bmatrix}\be_1&\be_{k+1}&\be_2&\be_{k+2}&\cdots&\be_{k}&\be_{2k}\end{bmatrix},
\]
we then  have that
$P^T R_{2k}P\in\mathbb{C}^{2k\times 2k}$ is an  upper triangular matrix with its diagonal elements being
$\tau_{1}, \overline{\tau}_{1}, \tau_{2}, \overline{\tau}_{2}, \ldots, \tau_{k}, \overline{\tau}_{k}$, where
$\tau_{j}$ are the $j$-th diagonal entries of $R_t$ for $j=1, \ldots, k$.
Hence $R_{2k}$ is nonsingular as claimed. The second part result in  (\romannumeral1) can be verified by some simple calculations and we skip the details.

The results in  (\romannumeral2),  (\romannumeral3) can be verified directly, and we omit the proof here.

Considering  (\romannumeral4), by denoting $C_{k+1,k}\in\mathbb{R}^{(k+1)\times k}$ the matrix which
is strictly lower triangular with  its  sub-diagonal elements being $1$ and all others being $0$,
it then  follows from Lemma \ref{lemma-PiHl} that
\[
	\scrH K_{2k}(\scrH, \bq_1)=K_{2k+2}(\scrH, \bq_1)\diag(C_{k+1,k}, \ -C_{k+1,k}).
\]
By noting $K_{2k+2}(\scrH, \bq_1)=Q_{2k+2}R_{2k+2}$ and the result in (\romannumeral2)
we   hence have
\[
\scrH Q_{2k}=Q_{2k+2}R_{2k+2}\diag(C_{k+1,k}, \ -C_{k+1,k})R_{2k}^{-1}.
\]
Denote the matrix that collects the first $k$ rows of $C_{k+1,k}$ by $\widetilde{C}$. Then
by   the  structure of $R_{2k}^{-1}$  we get
\begin{align*}
M := R_{2k}\begin{bmatrix}\widetilde{C}&\\ & -\widetilde{C}\end{bmatrix}R_{2k}^{-1}
=\begin{bmatrix}M_{11}&-\overline{M}_{21}\\
&\\
M_{21}&-\overline{M}_{11}\end{bmatrix}
\end{align*}
with  
\begin{align*}
	M_{11}&=(R_t\widetilde{C}+\overline{R}_s\widetilde{C}\overline{R}_t^{-1}R_s)(R_t-\overline{R}_s\overline{R}_t^{-1}R_s)^{-1}, \\ 
	M_{21}&=(R_s\widetilde{C}+\overline{R}_t\widetilde{C}\overline{R}_t^{-1}R_s)(R_t-\overline{R}_s\overline{R}_t^{-1}R_s)^{-1},
\end{align*} 
respectively,  being  unreduced  upper Hessenberg and upper triangular, and
\[
\begin{bmatrix}\br_t&\overline{\br}_s\\ \br_s&\overline{\br}_t\\ \tau_{k+1}&0\\ 0&\overline{\tau}_{k+1}\end{bmatrix}
\begin{bmatrix}\be_k^T &\\&-\be_k^T \end{bmatrix}R_{2k}^{-1}
=\left[\begin{array}{cc}
L_{11}&\qquad \qquad  -\overline{L}_{21}\\
L_{21}&\qquad \qquad  -\overline{L}_{11}\\
\zeta_{kk}\tau_{k+1}\be_k^T &\qquad \qquad  \mathbf 0\\
\mathbf 0&\qquad \qquad  -\overline{\zeta}_{kk}\overline{\tau}_{k+1}\be_k^T
\end{array}\right]\in\mathbb{C}^{(2k+2)\times 2k},
\]
where $L_{11}=\begin{bmatrix}\mathbf 0&\cdots&\mathbf 0&\zeta_{kk}\br_t\end{bmatrix}\in\mathbb{C}^{k\times k}$,
$L_{21}=\begin{bmatrix}\mathbf 0&\cdots&\mathbf 0&\zeta_{kk}\br_s\end{bmatrix}\in\mathbb{C}^{k\times k}$.
Writing
\[
\widetilde{T}_{2k}=\begin{bmatrix}M_{11}&-\overline{M}_{21}\\ \mathbf 0&\mathbf 0\\
M_{21}&-\overline{M}_{11}\\ \mathbf0& \mathbf0\end{bmatrix} +
\left[\begin{array}{cc}
L_{11}&-\overline{L}_{21}\\
\zeta_{kk}\tau_{k+1}\be_k^T &\mathbf 0\\
L_{21}&-\overline{L}_{11}\\
\mathbf 0& -\overline{\zeta}_{kk}\overline{\tau}_{k+1}\be_k^T
\end{array}
\right],
\]
then after  some simple calculations it
shows that  $\scrH Q_{2k}=Q_{2k+2}\widetilde{T}_{2k}$ as stated in (\romannumeral4).
Furthermore,   since
\begin{align}\label{Q_2k}
Q^{H}_{2k}\Gamma_0Q_{2k+2}=\begin{bmatrix}D_k&0&\mathbf 0&0\\ \mathbf 0&0&-D_k&0\end{bmatrix}\in\mathbb{C}^{2k\times (2k+2)},
\end{align}
it straightly  follows that
$Q_{2k}^{H}\Gamma_{0}\scrH Q_{2k}=\Gamma_{2k}(M+L)$ by defining
$L:=\begin{bmatrix}L_{11}&-\overline{L}_{21}\\L_{21}&-\overline{L}_{11}\end{bmatrix}$.
On the other hand, Lemma \ref{lemma-H-hermitian} demonstrates  that
$Q_{2k}^{H}\Gamma_{0}\scrH Q_{2k}$ is Hermitian, leading to $(M^{H}+L^{H})\Gamma_{2k}=\Gamma_{2k}(M+L)$,
implying that $(M_{11}+L_{11})$ is  an unreduced  tridiagonal matrix and $(M_{21}+L_{21})$ is diagonal.
Thus, it  holds that $D_k(M_{11}+L_{11})$ is unreduced  Hermitian tridiagonal,
that is the result in (\romannumeral5).

For the results in  (\romannumeral6), they  directly follow from equation \eqref{Q_2k}.
\end{proof}

Theorem \ref{theorem-K-2k-2k+2} illustrates the relationship between the
\GQR \ factorization of $K_{2k}(\scrH, \bq_1)$ and $K_{2k+2}(\scrH, \bq_1)$ for any
initial vector $\bq_1\in\mathbb{C}^{2n}$, provided that the
\GQR \ factorization of $K_{2k+2}(\scrH, \bq_1)$ exists.
In particular, we derive that
$\scrH Q_{2k}=Q_{2k+2}\widetilde{T}_{2k}$ for some $\widetilde{T}_{2k}\in\mathbb{C}^{(2k+2)\times 2k}$,
a formula analogous to the one for the classical Lanczos method.
Theorem~\ref{theorem-krylov} in below  further demonstrates that the inverse of Theorem \ref{theorem-K-2k-2k+2} is also valid.

\begin{theorem}\label{theorem-krylov}
Let $Q_{2k+2}\in\mathbb{C}^{2n\times (2k+2)}$ and $Q_{2k}\in\mathbb{C}^{2n\times 2k}$ are of
the forms specified in \eqref{Q_2k+2} and there exists some matrix 
$\Gamma_{2k+2}\in\bGamma_{2k+2}$ such that
$$Q^{H}_{2k+2}\Gamma_0Q_{2k+2}=\Gamma_{2k+2}.$$  
If 
\begin{align}\label{decomposition}
\scrH Q_{2k}=Q_{2k+2}\widetilde{T}_{2k},
\end{align}
where  $\widetilde{T}_{2k}\in \mathbb{C}^{(2k+2)\times 2k}$ 
is in the form of \eqref{T_2k} with its   sub-matrices $T_{11}, T_{21}\in\mathbb{C}^{k\times k}$ 
respectively being    unreduced  tridiagonal and  diagonal.
Then for the first column $\bq_1\in\mathbb{C}^{2n}$ of $Q_{2k+2}$
(also the first column of $Q_{2k}$)
we  have the following   \GQR\ factorization   corresponding to
$K_{2k+2}(\scrH,  \bq_1)$:
\[
K_{2k+2}(\scrH, \bq_1)=Q_{2k+2}R_{2k+2},
\]
where $R_{2k+2}\in\mathbb{U}^+_{2k+2}$ is nonsingular.
Moreover, $\rank(K_{2k+2}(\scrH,  \bq_1))=(2k+2)$.
\end{theorem}
\begin{proof}
Write the $j$-th column of $Q_{2k+2}$ as $\bq_j$ for $j=1, \ldots, (k+1)$, suggesting that the
 $(k+1+j)$-th column of $Q_{2k+2}$   is $\Pi\overline{\bq}_j$,
 and  constitute $Q_{2j}\in \mathbb{C}^{2n\times 2j}$ as
\[
Q_{2j}=\begin{bmatrix} \bq_1&\bq_2&\cdots&\bq_{j}
&\Pi\overline{\bq}_1&\Pi\overline{\bq}_2&\cdots&\Pi\overline{\bq}_{j}\end{bmatrix}
\]
for $j=1, 2, \ldots, k$. For $j=2, \ldots, k-1$ define $S^{(j)}\in\mathbb{C}^{2(j+1)\times 2j}$  as
\[
S^{(j)}=\begin{bmatrix}T_{11}(1:j+1,1:j)&-\overline{T}_{21}(1:j+1,1:j)\\
T_{21}(1:j+1,1:j)&-\overline{T}_{11}(1:j+1,1:j)\end{bmatrix},
\]
it then  follows from $\scrH Q_{2k}=Q_{2k+2}\widetilde{T}_{2k}$ that $\scrH Q_{4}=Q_{6}S^{(2)}$.
Furthermore, by induction, it is simple to get $\scrH^{j}Q_{4}=Q_{2(j+2)}S^{(j+1)}S^{(j)}\cdots S^{(2)}$
for all $j=1, \ldots, k-1$. By Lemma~\ref{lemma-PiHl} and the facts that
\[
\scrH\bq_1=Q_4
\begin{bmatrix}\alpha_1&\gamma_1&\eta_1&0\end{bmatrix}^T , \qquad
\scrH\Pi\overline{\bq}_1=Q_4
\begin{bmatrix}-\overline{\eta}_1&0&-\overline{\alpha}_1&-\overline{\gamma}_1\end{bmatrix}^T
\]
with $\alpha_1, \gamma_1$ respectively  
being the corresponding  $(1,1)$, $(2,1)$ elements of $T_{11}$ and
$\eta_1$ being the  $(1,1)$ one of $T_{21}$, we  obtain
\begin{equation}\label{krylov}
\begin{array}{l}
\scrH^{j+1}\bq_1=Q_{2(j+2)}S^{(j+1)}S^{(j)}\cdots S^{(2)}
\begin{bmatrix}\alpha_1&\gamma_1&\eta_1&0\end{bmatrix}^T , \\
\Pi\overline{\scrH}^{j+1}\overline{\bq}_1
=(-1)^{j+1}Q_{2(j+2)}S^{(j+1)}S^{(j)}\cdots S^{(2)}
\begin{bmatrix}-\overline{\eta}_1&0&-\overline{\alpha}_1&-\overline{\gamma}_1\end{bmatrix}^T .
\end{array}
\end{equation}
Additional, it holds that
\begin{align}\label{p0q0}
\bq_{1}=Q_2\begin{bmatrix}1&0\end{bmatrix}^T , & \qquad
\Pi\overline{\bq}_1=Q_2\begin{bmatrix}0&1\end{bmatrix}^T ,
\end{align}
and
\begin{align}\label{Hq0p0}
\Pi\overline{\scrH}\overline{\bq}_1=-\scrH\Pi\overline{\bq}_1.
\end{align}
Now by \eqref{krylov}, \eqref{p0q0} and \eqref{Hq0p0},  some
simple computations lead to
$K_{2k+2}(\scrH, \bq_1)=Q_{2k+2}R_{2k+2}$ with
\[
R_{2k+2}=\begin{bmatrix}R_t&\overline{R}_s\\R_s&\overline{R}_t\end{bmatrix},
\]
where $R_t, R_s\in\mathbb{R}^{(k+1)\times (k+1)}$ both are upper  triangular.
Moreover, being aware that the last row of $S^{(j+1)}S^{(j)}\cdots S^{(2)}$ is
$\begin{bmatrix}0&0&0&\varrho\end{bmatrix}^T $  for some $\varrho\in\mathbb{C}$,
hence $R_s$ actually is strictly upper triangular.
Additionally, since for $j=1, \ldots, (k+1)$, the $j$-th diagonal elements of
$R_{t}$ are $(\gamma_{0}\gamma_1\cdots\gamma_{j-1})$
with $\gamma_0=1$ and $\gamma_j$ being the $j$-th sub-diagonal entries of $T_{11}$,
we have that  $R_{2k+2}$ is nonsingular, which is equivalent to the result we are to prove.
\end{proof}

The decomposition specified in \eqref{decomposition} is named as a \GLan\   decomposition.
When  amalgamating  Theorem~\ref{theorem-K-2k-2k+2} and Theorem~\ref{theorem-krylov}, we
 have the  result concluded in Theorem~\ref{thm:implicity_theorem}, which
explicitly illuminates that the \GLan\   decomposition  \eqref{decomposition}
is essentially unique for the given initial vector $\bq_1\in\mathbb{C}^{2n}$.

\begin{theorem}[Implicity $\mathrm{\Gamma}$-orthogonality Theorem]\label{thm:implicity_theorem}
Let $\scrH Q_{2k}=Q_{2k+2}\widetilde{T}_{2k}$ and $\scrH\widehat{Q}_{2k}=\widehat{Q}_{2k+2}\widehat{T}_{2k}$
be two \GLan\   decompositions, where   $\widetilde{T}_{2k}, \widehat{T}_{2k}\in\mathbb{C}^{(2k+2)\times 2k}$
are in the form of \eqref{T_2k}, $Q_{2k+2}, \widehat {Q}_{2k+2}$ are  $\mathrm{\Gamma}$-orthogonal  with respect to $\Gamma_0$,
i.e., there exist two matrices $\Gamma_{2k+2}, \widehat{\Gamma}_{2k+2}\in\bGamma_{2k+2}$, such that
$Q_{2k+2}^{H}\Gamma_0Q_{2k+2}=\Gamma_{2k+2}$ and   $\widehat{Q}_{2k+2}^{H}\Gamma_0\widehat{Q}_{2k+2}=\widehat{\Gamma}_{2k+2}$.
Provided that $Q_{2k+2}\be_1=\widehat{Q}_{2k+2}\be_1$, then it  holds that
$\widehat{Q}_{2k+2}=Q_{2k+2}\diag(\Theta, \overline{\Theta}) \Gamma_{2k+2}$ and
\begin{align}\label{eq:hatT_tildeT}
	\widetilde{T}_{2k}=\diag(\Theta, \overline{\Theta})
\Gamma_{2k+2}\widehat{T}_{2k}\diag(\Theta(1:k,1:k), \overline{\Theta}(1:k,1:k))^{-1}\Gamma_{2k},
\end{align}
where  $\Theta=\diag(\theta_1, \ldots, \theta_{k+1})\in\mathbb{C}^{(k+1)\times (k+1)}$ with
$|\theta_j|=1$ for $j=1, \cdots, k+1$
and $\Gamma_{2k}=\diag(\Gamma_{2k+2}(1:k,1:k), -\Gamma_{2k+2}(1:k, 1:k))$.
\end{theorem}

\begin{proof}
It follows from Theorem \ref{theorem-krylov} that there are two
nonsingular matrices $R_{2k+2}, \widehat{R}_{2k+2}\in\mathbb{C}^{(2k+2)\times (2k+2)}$ such that
\[
K_{2k+2}(\scrH,Q_{2k+1}\be_1)=Q_{2k+2}R_{2k+2}=\widehat{Q}_{2k+2}\widehat{R}_{2k+2},
\]
where $R_{2k+2}$ and
$\widehat{R}_{2k+2}$ both are of the form specified in \eqref{R-2k+2}, leading to
\[
(R_{2k+2}^{-1})^{H}\widehat{R}_{2k+2}^{H}\widehat{\Gamma}_{2k+2}=\Gamma_{2k+2}R_{2k+2}\widehat{R}_{2k+2}^{-1}.
\]
Since $R_{2k+2}^{-1}$ and $\widehat{R}_{2k+2}^{-1}$ are of the form illustrated in \eqref{R-2k+2},
it then holds that $\Gamma_{2k+2}R_{2k+2}\widehat{R}_{2k+2}^{-1}=\diag(\Theta, \overline{\Theta})$ with
$\Theta\in\mathbb{C}^{(k+1)\times (k+1)}$ being some diagonal matrix,
suggesting that $\widehat{Q}_{2k+2}=Q_{2k+2}\diag(\Theta, \overline{\Theta})\Gamma_{2k+2}$.
Furthermore, by noticing the relationship between   $Q_{2k}$ and $Q_{2k+2}$
(a same one between   $\widehat{Q}_{2k}$ and $\widehat{Q}_{2k+2}$),
we get the result \eqref{eq:hatT_tildeT}.
\end{proof}

The \GLan\   decomposition proposed in the subsection essentially   is
a biorthogonal Lanczos procedure \cite{saad2011numerical}.
The reason is as follows. Since the decomposition \eqref{decomposition} is equivalent to
$(Q_{2k}^{H}\Gamma_0) \scrH=\widetilde{T}_{2k}^{H}(Q_{2k+2}^{H}\Gamma_0)$, then
pre-multiplying $\Gamma_{2k}$ and
post-multiplying $Q_{2k+2}$ on both sides of the above equation yield
\begin{align}
(\Gamma_{2k}Q_{2k}^{H}\Gamma_0) \scrH Q_{2k+2}
&=\Gamma_{2k}\widetilde{T}_{2k}^{H}(Q_{2k+2}^{H}\Gamma_0Q_{2k+2})
	=\Gamma_{2k}\widetilde{T}_{2k}^{H}\Gamma_{2k+2} \nonumber
	\\
	&\equiv
	\begin{bmatrix}
		T_{11} & * & -\overline{T}_{21} & \mathbf 0 \\
		T_{21} & \mathbf 0 & -\overline{T}_{11} & * 
	\end{bmatrix},\label{lanczos-w}
\end{align}
where $T_{11}$ and $T_{21}$ are the matrices defined in \eqref{T_2k}, 
leading to $(Q_{2k}^{H}\Gamma_0) \scrH (Q_{2k}\Gamma_{2k})=T_{2k}^{H}$,
which also is equivalent to
\begin{align}\label{lanczos-whv}
Q_{2k}^{H} \scrH^{H} (\Gamma_0Q_{2k}\Gamma_{2k})=T_{2k}^{H},
\end{align}
where
$T_{2k}$ is  the  submatrix  of $\widetilde{T}_{2k}$
constructed by the first $k$  and $(k+2)$-th, $\ldots$, $(2k+1)$-th rows of $\widetilde{T}_{2k}$.
Besides, \eqref{lanczos-w} gives
\begin{align}\label{lanczos-hw}
	\scrH^{H}(\Gamma_0Q_{2k}\Gamma_{2k})=(\Gamma_0Q_{2k+2})\widetilde{T}_{2k}\Gamma_{2k}
	=(\Gamma_0Q_{2k+2}\Gamma_{2k+2})(\Gamma_{2k+1}\widetilde{T}_{2k}\Gamma_{2k}).
\end{align}
Additional, it holds that $Q_{2k}^{H}\Gamma_0Q_{2k}\Gamma_{2k}=I_{2k}$.
Consequently, those equations \eqref{decomposition}, \eqref{lanczos-hw} and \eqref{lanczos-whv}
together reveal that the whole \GLan\   process is a biorthogonal Lanczos procedure.

\subsection{\GLan\   algorithm}
From subsection \ref{subsection4.1} we know that for the given initial vector
$\bq_1\in\mathbb{C}^{2n}$ with
\begin{align}\label{initial}
\bq_1^{H}&\Gamma_0\bq_1=1,
\end{align}
which obviously satisfies  $\bq_1^{H}\Gamma_0\Pi\overline{\bq}_1=0$,
once  one $Q$-factor $Q_{2k}$ of $K_{2k}(\scrH,\bq_1)$, which corresponds to a  \GQR\ factorization of $K_{2k}(\scrH,\bq_1)$,
is acquired by computing a factorization
specified in \eqref{decomposition}, 
we then can obtain  a $\mathrm{\Gamma}$-orthogonal basis for  the $\Pi$-Krylov subspace $\mathcal{K}_{2k}(\scrH, \bq_1)$.
In this subsection, we will show how to get  the decomposition \eqref{decomposition} for
the given initial vector $\bq_1\in\mathbb{C}^{2n}$ satisfying the initial condition \eqref{initial}.

Write the $j$-th  column of $Q_{2k+2}$ as $\bq_j$,
implying  the $(k+1+j)$-th column of $Q_{2k+2}$  is  $\Pi\overline{\bq}_j$,
and collect the first $k$ and the $(k+2)$-th, $\ldots$, $(2k+1)$-th columns of $Q_{2k+2}$ in
$Q_{2k}$, i.e.,
\begin{align*}
Q_{2k+2}&=\begin{bmatrix}\bq_1&\cdots&\bq_{k+1}&\Pi\overline{\bq}_1&\cdots&\Pi\overline{\bq}_{k+1}\end{bmatrix},
	\\
Q_{2k}&=\begin{bmatrix}\bq_1&\cdots&\bq_{k}&\Pi\overline{\bq}_1&\cdots&\Pi\overline{\bq}_{k}\end{bmatrix}.
\end{align*}
Let $\Gamma_{2k+2}=\diag(D_{\Gamma}, -D_{\Gamma})\in\mathbb{C}^{(2k+2)\times (2k+2)}$ be  the diagonal  matrix
satisfying
\begin{align}\label{orthogonal}
Q_{2k+2}^{H}\Gamma_0Q_{2k+2}=\Gamma_{2k+2}
\end{align}
and denote its  diagonal entries
by $\delta_1, \delta_2, \ldots, \delta_{k+1}, -\delta_1, -\delta_2, \ldots, -\delta_{k+1}$.
Obviously, $\delta_1=1$ due to the initial condition \eqref{initial}.
Since theorem \ref{theorem-K-2k-2k+2}  illustrates that 
$\diag(\delta_1, \ldots, \delta_k)\widetilde{T}_{2k}(1:k,1:k)$
is  a Hermitian matrix, we consequently can  write
the submatrices $\widetilde{T}_{2k}(1:k+1,1:k)$ and $\widetilde{T}_{2k}(k+2:2(k+1), 1:k)$ as
\begin{align}
&\widetilde{T}_{2k}(1:k+1,1:k)=\left[\begin{array}{cccccc}\alpha_1&\frac{\delta_2}{\delta_1}\overline{\beta}_1&&&\\
                  \beta_1&\alpha_2&\frac{\delta_3}{\delta_2}\overline{\beta}_2&&&\\
                  &\beta_2&\alpha_3&\frac{\delta_4}{\delta_3}\overline{\beta}_3&&\\
                  &&\ddots&\ddots&\ddots&\\
                  &&&\beta_{k-2}&\alpha_{k-1}&\frac{\delta_{k}}{\delta_{k-1}}\overline{\beta}_{k-1}\\
                  &&&&\beta_{k-1}&\alpha_k\\ \label{eq-T11}
                  &&&&&\beta_k\end{array}\right]\in\mathbb{C}^{(k+1)\times k}, \\
&\widetilde{T}_{2k}(k+2:2(k+1),1:k)=\left[\begin{array}{cccc}
                       \gamma_1&&&\\&\gamma_2&&\\ &&\ddots&\\ &&&\gamma_k \\ & & & 0\end{array}\right]
                       \in\mathbb{C}^{(k+1)\times k}, \label{eq-T21}
\end{align}
respectively, where $\alpha_j\in\mathbb{R}$ and $\beta_j, \gamma_j\in\mathbb{C}$ for $j=1, \ldots, k$.

Apparently, to get the decomposition \eqref{decomposition} for
the given vector $\bq_1\in\mathbb{C}^{2n}$, we just need to compute $Q_{2k+2}$ and
those $\alpha_j, \beta_j, \gamma_j, \delta_j$.
Analogously to the classical Lanczos method,
we can compute   $Q_{2k+2}$ and $\widetilde{T}_{2k}$ column by column.
The details are as follows.

Now regarding  the first column of \eqref{decomposition}, which is
\begin{align}\label{decom-firstcolumn}
\scrH\bq_1=\alpha_1\bq_1+\beta_1\bq_2+\gamma_1\Pi\overline{\bq}_1,
\end{align}
by sequentially  pre-multiplying $\bq_1^{H}\Gamma_0$ and
$(\Pi\overline{\bq}_1)^{H}\Gamma_0$ on both sides
of  \eqref{decom-firstcolumn}  and
noting the orthogonal prerequisite \eqref{orthogonal} we should set
\begin{align}\label{alpha1eta1}
\alpha_1=\bq_1^{H}\Gamma_0\scrH\bq_1\in\mathbb{R},
& \qquad \gamma_1=-(\Pi\overline{\bq}_1)^{H}\Gamma_0\scrH\bq_1.
\end{align}
Moreover, by defining
\begin{align}\label{z1}
\bz_2=\scrH\bq_1-\alpha_1\bq_1-\gamma_1\Pi\overline{\bq}_1,
\end{align}
it is easy to check that  $\bz_2^{H}\Gamma_0\bq_1=0$,  $\bz_2^{H}\Gamma_0(\Pi\overline{\bq}_1)=0$
and $\bz_2^{H}\Gamma_0\bz_2$  is a real number.
Additional, $\bz_2^{H}\Gamma_0\Pi_n\overline{\bz}_2=0$ is trivial.
Clearly, by setting $\bz_2=\beta_1\bq_2$ (satisfying $\bq_2^{H}\Gamma_0(\Pi\overline{\bq}_2)=0$), 
since  $\bq_2\in\mathbb{C}^{2n}$ should satisfy $\bq_2^{H}\Gamma_0\bq_2=\delta_2$ 
and  $\bz_2^{H}\Gamma_0\bz_2=|\beta_1|^2\delta_2$, we get  
\begin{align}\label{signd2beta1}
\delta_2=\sign(\bz_2^{H}\Gamma_0\bz_2).
\end{align}
Consequently,  we can take
\begin{align}\label{beta1p1q1}
\beta_1=\sqrt{|\bz_2^{H}\Gamma_0\bz_2|}, \qquad
\bq_2=\frac{1}{\beta_1}\bz_2.
\end{align}

%

Provided that we have acquired the first $j$ columns of  $\widetilde{T}_{2k}$, the first $(j+1)$ columns of
$Q_{2k}$, and $\delta_1, \ldots, \delta_{j+1}$  with $j\geq 1$,
implicating that the $(k+1)$-th, $\ldots$, $(k+j+1)$-th columns of $Q_{2k}$ have already been computed,
then we  are to  calculate  the $(j+1)$-th column of $\widetilde{T}_{2k}$, the $(j+2)$-th column of $Q_{2k}$ and also the scalar $\delta_{j+2}$.

Upon comparing the $(j+1)$-th column of \eqref{decomposition}, it holds that
\begin{align}\label{decom-j+1column}
\scrH\bq_{j+1}
=\frac{\delta_{j+1}}{\delta_j}\overline{\beta}_{j}\bq_{j}+
\alpha_{j+1}\bq_{j+1}+
\beta_{j+1}\bq_{j+2}+
\gamma_{j+1}(\Pi\overline{\bq}_{j+1}).
\end{align}
Similarly to the process to get the first column  of  $\widetilde{T}_{2k}$ and the second column of $Q_{2k}$,
it  is necessary  to take
\begin{align}\label{alphajetaj}
\left\{
\begin{array}{l}
\alpha_{j+1}=\delta_{j+1}\bq_{j+1}^{H}\Gamma_0\scrH\bq_{j+1}, \\
\gamma_{j+1}=-\delta_{j+1}(\Pi\overline{\bq}_{j+1})^{H}\Gamma_0\scrH\bq_{j+1},
\end{array}
\right.
\end{align}
due to the $\mathrm{\Gamma}$-orthogonal postulation $Q_{2k}^{H}\Gamma_0Q_{2k}=\Gamma_{2k}$. 
For $\alpha_{j+1}$ in \eqref{alphajetaj},  it holds that $\alpha_{j+1}\in\mathbb{R}$.
Additional, let
\begin{align}\label{zj+2}
\bz_{j+2}=\scrH\bq_{j+1}
-\frac{\delta_{j+1}}{\delta_j}\overline{\beta}_{j}\bq_{j}
-\alpha_{j+1}\bq_{j+1}
-\gamma_{j+1}(\Pi\overline{\bq}_{j+1}),
\end{align}
then analogously to the vector $\bz_2$, $\bz_{j+2}$ enjoys  some  properties
listed in  Lemma~\ref{lemmazj}.
\begin{lemma}\label{lemmazj}
For $\bz_{j+2}\in\mathbb{C}^{2n}$ it holds that
\begin{enumerate}
  \item[{\em (\romannumeral1)}] $\bz_{j+2}^{H}\Gamma_0\bq_l=0$ and
                   $\bz_{j+2}^{H}\Gamma_0(\Pi\overline{\bq}_l)=0$
                   for all $l=1, \ldots, (j+1)$; and 
  \item[{\em (\romannumeral2)}] $\bz_{j+2}^{H}\Gamma_0\bz_{j+2}$ is a real number. 
\end{enumerate}
\end{lemma}
\begin{proof}
Firstly, some directly calculations show that
\begin{align*}
\bz_{j+2}^{H}\Gamma_0\bq_l=
\begin{cases}
\bq_{j+1}^{H}\scrH^{H}\Gamma_0\bq_l,
&\qquad l=1, \ldots, (j-1);\\
\bq_{j+1}^{H}\scrH^{H}\Gamma_0\bq_{j}
-\delta_{j+1}\overline{\beta}_{j}, &\qquad l=j;\\
\bq_{j+1}^{H}\scrH^{H}\Gamma_0\bq_{j+1}
-\delta_{j+1}\alpha_{j+1},& \qquad l=j+1.
\end{cases}
\end{align*}
Moreover, Lemma \ref{lemma-H-hermitian}, the decomposition \eqref{decomposition} and
the definition $\alpha_{j+1}$ specified in \eqref{alphajetaj} will  cooperatively certificate that
\begin{align*}
	&\bq_{j+1}^{H}\scrH^{H}\Gamma_0\bq_l=0 \qquad \text{for} \qquad l=1, \ldots, (j-1), \\
	&\bq_{j+1}^{H}\scrH^{H}\Gamma_0\bq_{j}=\delta_{j+1}\overline{\beta}_j, \qquad
\bq_{j+1}^{H}\scrH^{H}\Gamma_0\bq_{j+1}=\delta_{j+1}\alpha_{j+1},
\end{align*}
leading to the result we are to prove. And for the second result in (\romannumeral1),
it  can be proved in the same way,  we hence omit the details here.
For the  result in (\romannumeral2), it holds that
$\bz_{j+2}^{H}\Gamma_0\bz_{j+2}=\bz_{j+2}^{H}\Gamma_0 \scrH\bq_{j+1}
=\bq_{j+1}^{H}\scrH^{H} \Gamma_0 \scrH \bq_{j+1}
-\frac{1}{\delta_{j}}|\beta_j|^2-\frac{1}{\delta_{j+1}}|\alpha_{j+1}|^2+\frac{1}{\delta_{j+1}}|\gamma_{j+1}|^2$,
which is a real scalar.
\end{proof}

Lemma \ref{lemmazj}  fundamentally  guarantees no interruptions of the computing procedure for 
the $(j+2)$-th column $\bq_{j+2}$ in $Q_{2k}$ and
the $(j+1)$-th subdiagonal  entry  $\beta_{j+1}$ in  $\widetilde{T}_{2k}$. To this end,
we set $\beta_{j+1}\bq_{j+2}=\bz_{j+2}$.
On the other hand, $\bq_{j+2}^{H}\Gamma_0\bq_{j+2}=\delta_{j+2}$
gives  $\bz_{j+2}^{H}\Gamma_0 \bz_{j+2}=|\beta_{j+1}|^2\delta_{j+2}$,
suggesting
\begin{align}\label{dj}
\delta_{j+2}=\sign(\bz_{j+2}^{H}\Gamma_0 \bz_{j+2}).
\end{align}
Furthermore, $\beta_{j+1}$ can be arranged as a real scalar, that is,
\begin{align}\label{betaj}
\beta_{j+1}=\sqrt{|\bz_{j+2}^{H}\Gamma_0 \bz_{j+2}|},
\end{align}
and $\bq_{j+2}$ can be obtained as
\begin{align}\label{pqj}
\bq_{j+2}=\frac{1}{\beta_{j+1}}\bz_{j+2}.
\end{align}

It is worthwhile to point out  that the computing process of
the \GLan\   decomposition  \eqref{decomposition}  
can proceed without hindrance only when no subdiagonal elements $\beta_j$ equal to zero.
Otherwise, breakdown will happen. Without loss of generality, we assume that
$|\bz_{j+1}^{H}\Gamma_0 \bz_{j+1}|=0$ with
$\bz_{j+1}=\scrH\bq_{j}-\frac{\delta_{j}}{\delta_{j-1}}\overline{\beta}_{j-1}\bq_{j-1}
-\alpha_{j}\bq_{j}-\gamma_{j}(\Pi\overline{\bq}_{j})$, leading to  $\beta_{j}=0$,
then
\[
\scrH Q_{2j}=Q_{2j}T_{2j}, \qquad  T_{2j}=\begin{bmatrix}T_{11}&-\overline{T}_{21}\\T_{21}&-\overline{T}_{11}\end{bmatrix},
\]
where $T_{11}\in\mathbb{C}^{j\times j}$ is some tridiagonal matrix and $T_{21}\in\mathbb{C}^{j\times j}$
is some diagonal matrix. By noting that $Q_{2j}^{H}\Gamma_0Q_{2j}=\Gamma_{2j}$ for some matrix 
$\Gamma_{2j}\in\bGamma_{2j}$, it then  holds that
\[
(\Gamma_{2j}Q_{2j}^{H}\Gamma_0)\scrH Q_{2j}=T_{2j}.
\]
Lemma \ref{Lemma-Q1} in the subsequent subsection demonstrates that $Q_{2j}$ can be expanded to a 
nonsingular matrix $Q\in\mathbb{C}^{2n\times 2n}$ with $Q=\begin{bmatrix}Q_{2j}&Q_{u}\end{bmatrix}$,  
	such that $Q^{H}\Gamma_0Q=\diag(\Gamma_{2j}, \Gamma_{2n-2j})$ for some $\Gamma_{2n-2j}\in\bGamma_{2n-2j}$.
Accordingly, we have $Q^{-1}=\diag(\Gamma_{2j}, \Gamma_{2n-2j})Q^{H}\Gamma_0$ and
\[
Q^{-1}\scrH Q=\begin{bmatrix}T_{2j}&\Gamma_{2j}Q_{2j}^{H}\Gamma_0\scrH Q_{u}\\
0&\Gamma_{2n-2j}Q_{u}^{H}\Gamma_0\scrH Q_u\end{bmatrix},
\]
implying that $\lambda(T_{2j})\subseteq \lambda(\scrH)$.
Consequently, we can acquire some eigenvalues of $\scrH$ by computing all eigenvalues of $T_{2j}$, which is of small size.

We synthesize the procedure  of computing
the \GLan\   decomposition  \eqref{decomposition} in algorithm \ref{algorithm_lanczos}.

\begin{algorithm}[t]
\caption{\GLan\   Algorithm}
\label{algorithm_lanczos}
\begin{algorithmic}[1]
   \REQUIRE $\scrH$, $\Gamma_0$, initial vector $q_1\in\mathbb{C}^{2n}$ satisfying \eqref{initial} and $k$;
   \ENSURE  $Q_{2k}\in\GO_{2n\times 2k}$ with respect to $\Gamma_0$, $\Gamma_{2k}=Q_{2k}^H\Gamma_0 Q_{2k}\in\bGamma_{2k}$,
            $\widetilde{T}_{2k}\in\mathbb{C}^{(2k+2)\times 2k}$ being the form specified in \eqref{eq-T11} and \eqref{eq-T21}.
   \vspace{1ex}\hrule\vspace{1ex}
   \STATE Set $\Gamma_{2k}(1,1)=1$ and $\Gamma_{2k}(k+1, k+1)=-1$.
   \STATE Compute $\alpha_1$ and $\gamma_1$ by \eqref{alpha1eta1}.
   \STATE Compute $\delta_2$, $\beta_1$ and $\bq_2$ by \eqref{signd2beta1} and \eqref{beta1p1q1}, respectively.
   \STATE Set j=2;
   \FOR{$j=2:1:k-1$}
     \STATE  Compute $\alpha_j$ and $\gamma_j$ by \eqref{alphajetaj}.
     \STATE  Compute $\delta_{j+1}$, $\beta_j$ and $\bq_{j+1}$ by \eqref{dj}, \eqref{betaj} and \eqref{pqj}, respectively.
     \ENDFOR
   \STATE Compute $\bz_{k+1}\in\mathbb{C}^{2n}$ by \eqref{zj+2}.
   \STATE Compute $\alpha_k$, $\gamma_k$ and $\beta_k$ by \eqref{alphajetaj} and \eqref{betaj}, respectively.
\end{algorithmic}
\end{algorithm}


\begin{remark}\label{rm:remarklancos}
Here, \GLan~Algorithm \ref{algorithm_lanczos} focuses on computing some extreme eigenpais of $\scrH$, 
i.e., those eigenvalues with maximum and minimum moduli and their associated  eigenvectors. 
However, if some eigenvalues around  $\sigma\in\mathbb{C}$ are needed, 
 the double-shift-invert technique and the \GLan~algorithm will cooperate  
 to  catch them and also their associated  eigenvectors. In detail,  define  
\[
	\mathscr{W}=(\scrH-\sigma I_{2n})(\scrH+\sigma I_{2n})\equiv\scrH^2-\sigma^2 I_{2n},
\]
and
\[
	\mathscr{W}= (\scrH-\sigma I_{2n})(\scrH+\sigma I_{2n})(\scrH-\bar{\sigma} I_{2n})
	(\scrH+\bar{\sigma} I_{2n})
	\equiv \scrH^4-(\sigma^2+\overline{\sigma}^2)\scrH + |\sigma|^4 I_{2n},  
\]
when  $\sigma$ is real / purely imaginary and complex. 
It is easily  checked that $\mathscr{W}$ is a $\Pi^+$-Hermitian matrix with respect to 
$\Gamma_0\in \bGamma_{2n}$. According to   Remark~\ref{remark3}, 
there exists a $\Gamma$-unitary matrix $Q$ such that $Q^{-1}\mathscr{W}Q$ is a 
$\Pi^{+}$-Hermitian-tridiagonal matrix with respect to some $\Gamma_{2n}\in\bGamma_{2n}$. 
As a result, the proposed \GLan~algorithm can be applied to $\mathscr{W}^{-1}$ implicitly 
for the computation of  any designed eigenpairs.
\end{remark}

\subsection{Error bounds}
\begin{lemma}\label{Lemma-Q1}
For any arbitrary $Q_1\in\GO_{2n\times 2k}$,
if $Q_1^{H}\Gamma_{2n}Q_1=\Gamma_{2k}$ for $\Gamma_{2n}\in\bGamma_{2n}$, $\Gamma_{2k}\in\bGamma_{2k}$
and write $\Gamma_{2k}=\diag(D_{\Gamma_{2k}},  -D_{\Gamma_{2k}})$,
then there exists some matrix $Q_{2}\in\GO_{2n\times (2n-2k)}$
such that $Q=\begin{bmatrix}Q_1&Q_2\end{bmatrix}\in\mathbb{C}^{2n\times 2n}$ is nonsingular and
	satisfies $Q^{H}\Gamma_{2n}Q=\diag(\Gamma_{2k}, \Gamma_{2n-2k})$ with $\Gamma_{2n-2k}\in\bGamma_{2n-2k}$.
\end{lemma}
\begin{proof}
Since for the two given matrices $Q_1$ and $\Gamma_{2n}$  there exists a \GQR\ factorization
of $Q_1$ with respect to $\Gamma_{2n}$, which is written as $Q_1=PR$ with 
$P\in\GO_{2n}$,  $R\in\mathbb{U}^{+}_{2n\times 2k}$, 
$P^{H}\Gamma_{2n}P=\widetilde{\Gamma}_{2n}$  for $\widetilde{\Gamma}_{2n}\in\bGamma_{2n}$,
it then  holds that $R^{H}\widetilde{\Gamma}_{2n}R=\Gamma_{2k}$,
which is equivalent to
\begin{align}\label{Q1QR}
R_1^{H}D_{\widetilde{\Gamma}_{2n}}R_1-R_2^{H}D_{\widetilde{\Gamma}_{2n}}R_2=D_{\Gamma_{2k}},  \qquad
R_2^T D_{\widetilde{\Gamma}_{2n}}R_1=R_1^T D_{\widetilde{\Gamma}_{2n}}R_2,
\end{align}
where $R=\begin{bmatrix}R_1&\overline{R}_2\\R_2&\overline{R}_1\end{bmatrix}$, $R_1, R_2\in \mathbb{C}^{n\times k}$
	and $\widetilde{\Gamma}_{2n}=\diag(D_{\widetilde{\Gamma}_{2n}},   -D_{\widetilde{\Gamma}_{2n}})$. By noticing the
structures of $R_1$ and $R_2$, which are upper triangular and strictly upper triangular, respectively,
we get that $R_1$ is diagonal and $R_2=0$ upon comparing each
entry of both sides of the two equations in \eqref{Q1QR}.
Furthermore, $D_{\widetilde{\Gamma}_{2n}}(1:k,1:k)=D_{{\Gamma}_{2k}}$ and the modulus  of all 
diagonal elements of $R_1$ equal to $1$, yielding that
$Q_1=\begin{bmatrix}P(:,1:k)\Theta&P(:,n+1:n+k)\overline{\Theta}\end{bmatrix}$ with
$\Theta=\diag(\theta_1, \ldots, \theta_k)$,  $|\theta_j|=1$ for $j=1, \ldots, k$.
Hence by taking
\[Q_2=\begin{bmatrix}P(:,k+1:n)&P(:,n+k+1:2n)\end{bmatrix},
\]
the result follows directly.
\end{proof}

Now applying  Lemma \ref{Lemma-Q1} to  the matrix $Q_{2k+2}\in\mathbb{C}^{2n\times (2k+2)}$
in \eqref{decomposition}, which is obtained via pursuing  
the  \GLan\   algorithm in the previous subsection,
 we then  get  there exists  $Q_{u}\in\GO_{2n\times (2n-2k-2)}$ with
$Q_{u}^{H}\Gamma_0Q_{u}=\Gamma_{2n-2k-2}$ for some $\Gamma_{2n-2k-2}\in\bGamma_{2n-2k-2}$,
such that
$Q:= \begin{bmatrix}Q_{2k+2}&Q_{u}\end{bmatrix}\in\mathbb{C}^{2n\times 2n}$ is nonsingular with
$Q^{-1}=\diag(\Gamma_{2k+2},  \Gamma_{2n-2k-2})Q^{H}\Gamma_0$ and
satisfies $Q^{H}\Gamma_0Q=\diag(\Gamma_{2k+2},  \Gamma_{2n-2k-2})$.
Pre-multiplying $Q^{-1}$ on both sides of \eqref{decomposition}
gives  $Q^{-1} \scrH Q_{2k}=
\begin{bmatrix}\widetilde{T}_{2k}^T &0\end{bmatrix}^T $.
Furthermore, let $P\in\mathbb{R}^{2n\times 2n}$ be the  permutation matrix with
\[
P=[\begin{array}{ccccccccccc}\be_{1} & \cdots & \be_{k} & \be_{k+2} & \cdots&
\be_{2k+1} & \be_{k+1} & \be_{2k+2} & \be_{2k+3} & \cdots & \be_{2n}\end{array}],
\]
we  then have
\begin{align}\label{HQP}
P^T Q^{-1}\scrH QP
=\left[\begin{array}{c|c|c}
\mathbf H_{11} &
\begin{array}{cc}\delta_{k}\delta_{k+1}\overline{\beta}_{k}\be_{k} &
-\delta_{k}\delta_{k+1}\beta_{k}\be_{2k}\end{array}
 & \mathbf 0 \\
&\\[-2mm]
\hline
&\\[-2mm]
\begin{array}{c}
\beta_{k}\be_k^T \\
-\overline{\beta}_k\be_{2k}^T \\
\end{array} & \mathbf H_{22}&\mathbf H_{23}\\
&\\[-2mm]
\hline
&\\[-2mm]
\mathbf 0 &\mathbf H_{32} &\mathbf H_{33}\end{array}\right]
\end{align}
with
\begin{align*}
	&\mathbf{H}_{11}=\begin{bmatrix}\widetilde{T}_{2k}(1:k,1:k)&-\overline{\widetilde{T}}_{2k}(k+2:2k+1,1:k)\\
\widetilde{T}_{2k}(k+2:2k+1,1:k)&-\overline{\widetilde{T}}_{2k}(1:k,1:k) \end{bmatrix}, \\
&\mathbf{H}_{22}=\begin{bmatrix}\delta_{k+1}&\\&-\delta_{k+1}\end{bmatrix}
\begin{bmatrix}\bq_{k+1}&\Pi\overline{\bq}_{k+1}\end{bmatrix}^{H}\Gamma_0 H
\begin{bmatrix}\bq_{k+1}&\Pi\overline{\bq}_{k+1}\end{bmatrix},\\
&\mathbf{H}_{23}=\begin{bmatrix}\delta_{k+1}&\\&-\delta_{k+1}\end{bmatrix}
\begin{bmatrix}\bq_{k+1}&\Pi\overline{\bq}_{k+1}\end{bmatrix}^{H}\Gamma_0 H Q_u,\\
	& \mathbf{H}_{32}=\Gamma_{2n-2k-2}Q_{u}^{H}\Gamma_0 H
\begin{bmatrix}\bq_{k+1}&\Pi\overline{\bq}_{k+1}\end{bmatrix},\qquad
	\mathbf{H}_{33}=\Gamma_{2n-2k-2}Q_{u}^{H}\Gamma_0 H Q_u,
\end{align*}
where $\delta_k$ and $\delta_{k+1}$, respectively,  are the $k$-th and $(k+1)$-th   
diagonal  entries of $\Gamma_{2k+2}$, 
and $\bq_{k+1}$ and $\Pi\overline{\bq}_{k+1}$ are  the $(k+1)$-th and $(2k+2)$-th columns of $Q_{2k+2}$.
Hence accordingly, in case that $\beta_{k}=0$, it will holds that
we can obtain some eigenvalues of $\scrH$ by computing those of $\mathbf{H}_{11}$.
More specifically, we will have the following theorem.

\begin{theorem}\label{theorem-rayleigh}
	For those  matrices   defined above, if $\mathbf{H}_{11}\bx =\nu \bx$ with
$0\neq \bx \in\mathbb{C}^{2k}$ and $\nu\in\mathbb{C}$, it then  holds that
\begin{enumerate}
	\item[{\em (\romannumeral1)}] $\|Q^{-1}(\scrH Q_{2k}-Q_{2k}\mathbf{H}_{11}) \|_2
		=\min\{\|Q^{-1}(\scrH Q_{2k}-Q_{2k}S) \|_2:
                    \forall S\in\mathbb{C}^{2k\times 2k}\}\equiv |\beta_{k+1}|$; and 
   \item[{\em (\romannumeral2)}] $\scrH Q_{2k} \bx-\nu Q_{2k} \bx =
                     \beta_k\xi_k\bq_{k+1}
                     -\xi_{2k}\overline{\beta}_{k}\Pi\overline{\bq}_{k+1}$,
                     where $\xi_k$ and $\xi_{2k}$ respectively  are the $k$-th and $2k$-th elements of $\bx$.
\end{enumerate}
\end{theorem}
\begin{proof}
Since $Q^{-1}\scrH Q_{2k}=PP^T Q^{-1}\scrH Q_{2k}=
P\begin{bmatrix}\mathbf{H}_{11}^T &\beta_{k}\be_{k}&-\overline{\beta}_{k}\be_{2k}&\bf 0\end{bmatrix}^T$
	due to \eqref{HQP},
where $\be_k, \be_{2k}\in\mathbb{C}^{2k}$ here,  it  holds that
\begin{align*}
&Q^{-1}(\scrH Q_{2k}-Q_{2k}S)\\
=&P\begin{bmatrix}\mathbf{H}_{11}^T &\beta_{k}\be_{k}&-\overline{\beta}_{k}\be_{2k}&\bf 0\end{bmatrix}^T
-Q^{-1}QP\begin{bmatrix}\be_1&\cdots&\be_{2k}\end{bmatrix}S\\
=&P\begin{bmatrix}(\mathbf{H}_{11}-S)^T &\beta_{k}\be_{k}&-\overline{\beta}_{k}\be_{2k}&\bf 0\end{bmatrix}^T ,
\end{align*}
leading to
\[
\|Q^{-1}(\scrH Q_{2k}-Q_{2k}S) \|_2=\|\begin{bmatrix}
(\mathbf{H}_{11}-S)^T &\beta_{k}\be_{k}&-\overline{\beta}_{k}\be_{2k}&0\end{bmatrix}^T \|_2
\geq |\beta_{k}|,
\]
which obviously achieves its minimum when $S=\mathbf{H}_{11}$.

For the result in  (\romannumeral2), it follows from the decomposition \eqref{decomposition} and the equation \eqref{HQP}
that
\[
	\scrH Q_{2k} \bx=Q_{2k}\mathbf{H}_{11}\bx+
\begin{bmatrix}\bq_{k+1}&\Pi\overline{\bq}_{k+1}\end{bmatrix}
\begin{bmatrix}\beta_{k}\be_k^T \\ -\overline{\beta}_k\be_{2k}^T \end{bmatrix}\bx,
\]
indicating that $\scrH Q_{2k} \bx-\nu Q_{2k} \bx=
\beta_k (\be_k^T \bx) \bq_{k+1}
-\overline{\beta}_{k}(\be_{2k}^T \bx)\Pi\overline{\bq}_{k+1}$.
Hence the result follows immediately.
\end{proof}

Theorem \ref{theorem-rayleigh} demonstrates that for the given matrix $Q\in\mathbb{C}^{2n\times 2n}$,
whose first $(2k+2)$ columns are computed by the \GLan\   algorithm, $\mathbf{H}_{11}$ will  be the
best candidate,  in some norm minimizing  sense,  with its eigenvalues to approximate those of $\scrH$.
Actually,  $\mathbf{H}_{11}$ acts as the Rayleigh quotient
for the matrix $\scrH$, similarly to that for  Hermitian matrices.
In addition,  the result in  (\romannumeral2) reveals that
$\|\scrH Q_{2k} \bx-\nu Q_{2k} \bx\|_2\leq
(|\beta_{k}\xi_{k}| +|\beta_{k}\xi_{2k}|)\|\bq_{k+1}\|_2$ and
$\|Q^{-1}(\scrH Q_{2k} \bx-\nu Q_{2k} \bx)\|_2\leq |\beta_{k}|\sqrt{|\xi_{k}|^2+|\xi_{2k}|^2}$,
illustrating in theory  how good it is when we use $(\nu, Q_{2k}x)$ to approximate one eigenpair of $\scrH$.

\subsection{Convergence theorem}
In this subsection, we will analyze the convergence rate of the $\bPi$-Krylov subspace in
Definition \ref{krylov-space} to the eigenspace of $\scrH$ corresponding to the
largest eigenvalues in absolute value.

Firstly, we introduce a definition for measuring the distance between two subspaces.
\begin{definition}(\cite{stsu:90})\label{def6}
Let $\mathcal{U}$ and $\mathcal{V}$ be two subspaces of $\mathbb{C}^n$ with dimensions $p$ and $l$, respectively, and $p\leq l$. We call the value
$$\text{dist}(\mathcal{U},\mathcal{V})=\max_{
	\substack{
      \sss\bu\in\mathcal{U}   \\
        \sss\|\bu\|_2=1
	}
} \min_{\sss\bv\in\mathcal{V}}\|\bu-\bv\|_2$$
as the {\em distance between $\mathcal{U}$ and $\mathcal{V}$}.
\end{definition}
For simplicity,  in what follows we always assume that all eigenvalues of $\scrH$
are semi-simple, and $\l_1=\rho e^{\imath \theta}$ is neither real nor purely imaginary with
\begin{equation}\label{eq6.1}
|\lambda_1|= |\l_2|= |\l_3|= |\l_4|>|\l_5|\geq \cdots\geq |\l_{2n}|.
\end{equation}
Let $\bx_1,\cdots,\bx_{2n}$ be the  eigenvectors of $\scrH$ corresponding to $\lambda_1, \ldots, \lambda_{2n}$, respectively,
where $\|\bx_j\|_2=1$ for $j=1,\cdots,2n$.
According to Lemma \ref{lemma-eigenpairs}, we can rearrange the eigenvalues in \eqref{eq6.1} as
\begin{equation}\label{eq6.2}
\lambda_2= -\l_1,\ \l_3= \overline{\l}_1,\ \l_4=-\overline{\l}_1\ \text{ and  }\
\bx_4=\Pi\overline{\bx}_1, \ \bx_3=\Pi\overline{\bx}_2.
\end{equation}
From the above assumption, it follows that  $\bx_1,\cdots,\bx_{2n}$ are linearly independent,
suggesting that the vector  $\bq_1$ in \eqref{K_2k} can be written  as a linear combination of all eigenvectors of $\scrH$:
\begin{equation}\label{eq6.3}
\bq_1=\sum_{j=1}^{2n}a_j\bx_j.
\end{equation}
Define
\begin{equation}\label{eq6.4}
X_1=[\bx_1,\bx_2,\bx_3,\bx_4] \quad  \text{ and} \quad \mathcal{X}_1=\text{span}\{\bx_1,\bx_2,\bx_3,\bx_4\}.
\end{equation}
Let $p_k$ and $\widehat{p}_k$ be two polynomials of degree $k$ in the following form (for convenience, we  suppose $k$ is even):
\begin{equation}\label{eq6.5}
p_k(z)=\zeta_1\left(\frac{z}{\l_1}\right)^k+\zeta_2\left(\frac{z}{\l_1}\right)^{k-1}, \quad  \widehat{p}_k(z)=\widehat{\zeta}_1\left(\frac{z}{\l_1}\right)^k+\widehat{\zeta}_2\left(\frac{z}{\l_1}\right)^{k-1}
\end{equation}
with $\bz \equiv\begin{bmatrix}\zeta_1&\zeta_2&\widehat{\zeta}_1&\widehat{\zeta}_2\end{bmatrix}^T \in\mathbb{C}^4$.
By \eqref{eq6.3} and Lemma \ref{lemma-eigenpairs}, we have
\begin{align}
	\bg &
	\equiv p_k(\scrH)\bq_1+\widehat{p}_k(\scrH)\Pi\overline{\bq}_1 \nonumber
\\
&=
\sum_{j=1}^{2n} a_j p_k(\scrH)\bx_j + 
\sum_{j=1}^{2n}\overline{a}_j \widehat{p}_k(\scrH)\Pi\overline{\bx}_j
\in \mathcal{K}_{2k+2}(\scrH,\bq_1).\label{eq6.6}
\end{align}
Extract the first four terms of \eqref{eq6.6} from $\bg$ as
 \begin{align}
\widehat{\bg}&=\sum_{j=1}^{4} a_j( p_k(\scrH)\bx_j + 
\sum_{j=1}^{4} \overline{a}_j \widehat{p}_k(\scrH)\Pi\overline{\bx}_j)
=
\sum_{j=1}^{4}a_j p_k(\l_j)\bx_j + 
\sum_{j=1}^4 \overline{a}_j \widehat{p}_k(-\overline{\l}_j)\Pi\overline{\bx}_j\nonumber\\
&= X_1\begin{bmatrix}  a_1p_k(\l_1)+\overline{a}_4\widehat{p}_k(\l_1)\\
	a_2p_k(-\l_1)+\overline{a}_3\widehat{p}_k(-\l_1)\\
	a_3p_k(\overline{\l}_1)+\overline{a}_2\widehat{p}_k(\overline{\l}_1)\\
	a_4p_k(-\overline{\l}_1)+\overline{a}_1\widehat{p}_k(-\overline{\l}_1)
 \end{bmatrix}\nonumber\\
 &=X_1\begin{bmatrix}  a_1&a_1&\overline{a}_4&\overline{a}_4\\
	 a_2&-a_2&\overline{a}_3&-\overline{a}_3\\
	 a_3e^{-\imath 2k\theta } &a_3e^{-\imath 2(k-1)\theta  } 
	 &\overline{a}_2e^{-\imath 2k\theta }&\overline{a}_2e^{-\imath 2(k-1)\theta  }\\
	 a_4e^{-\imath 2k\theta } &-a_4e^{-\imath 2(k-1)\theta  } 
	 &\overline{a}_1e^{-\imath 2k\theta  }&-\overline{a}_1e^{-\imath 2(k-1)\theta  }
 \end{bmatrix}\bz\nonumber\\
 &\equiv X_1 \Phi \bz \in\mathcal{X}_1. \label{eq6.8}
\end{align}
By denoting
\begin{equation}\label{eq6.9}E=\begin{bmatrix} 1&1\\1&-1 \end{bmatrix},\ F= \begin{bmatrix} 1&\\&\tau \end{bmatrix} \text{ with } \tau=e^{\imath2\theta},  \end{equation}
the matrix $\Phi$ in \eqref{eq6.8} can be rewritten as
\begin{equation}\label{eq6.10}
\Phi=\begin{bmatrix}
\begin{bmatrix} a_1\oplus a_2  \end{bmatrix} E &
\begin{bmatrix} \overline{a}_4 \oplus \overline{a}_3  \end{bmatrix} E\\
e^{-\imath 2k\theta } \begin{bmatrix} a_3 \oplus a_4  \end{bmatrix} EF &e^{-\imath 2k\theta } \begin{bmatrix} \overline{a}_2 \oplus \overline{a}_1  \end{bmatrix} EF
  \end{bmatrix}.
\end{equation}
If $a_1a_2\neq 0$, then $\Phi$ can be further decomposed into
\begin{equation}\label{eq6.11}
\Phi=
\begin{bmatrix} I_2 &\\ & e^{-\imath 2k\theta } \begin{bmatrix} \overline{a}_2 \oplus \overline{a}_1  
\end{bmatrix}
\end{bmatrix}
\begin{bmatrix}
 \begin{bmatrix} a_1 \oplus a_2  \end{bmatrix}
 &\b0 \\  \begin{bmatrix} \frac{a_3}{\overline{a}_2} \oplus \frac{a_4}{\overline{a}_1}  \end{bmatrix}
 EFE^{-1} & \widehat{\Phi}
\end{bmatrix}
\begin{bmatrix} E& \begin{bmatrix}  \frac{\overline{a}_4}{a_1} \oplus \frac{\overline{a}_3}{a_2}  
\end{bmatrix}E\\ \b0&  E\end{bmatrix} ,
\end{equation}
where
\begin{align}
\widehat{\Phi}
&=
EFE^{-1}-\begin{bmatrix}  \frac{a_3}{\overline{a}_2} \oplus \frac{a_4}{\overline{a}_1} 
\end{bmatrix}EFE^{-1}\begin{bmatrix}  \frac{\overline{a}_4}{a_1} \oplus \frac{\overline{a}_3}{a_2} 
\end{bmatrix}\nonumber
\\
&=\frac{1}{2}\begin{bmatrix} (1+\tau)(1- \frac{a_3\overline{a}_4}{a_1\overline{a}_2}) 
	&(1-\tau)(1- \frac{|a_3|^2}{|a_2|^2})\\  (1-\tau)(1- \frac{|a_4|^2}{|a_1|^2})&
(1+\tau)(1- \frac{\overline{a}_3a_4}{\overline{a}_1a_2})\end{bmatrix}.\label{eq6.12}
\end{align}
Consequently,  $\Phi$ is invertible if and only if $\widehat{\Phi}$ is nonsingular with the assumption $a_1a_2\neq 0$. Actually, based on the above analysis, we can get the same conclusion 
when  $a_3a_4\neq 0$. Thus, without loss of generality, we can  always require $a_1a_2\neq 0$ in \eqref{eq6.3}.

\begin{theorem}\label{thm6.1}
	Let  $\scrH$ have semi-simple eigenvalues
in the order of \eqref{eq6.1} and \eqref{eq6.2}.
Suppose $k$ (be even) steps of \GLan\ algorithm
 are performed with $\bq_1$ in the form of \eqref{eq6.3}.
 If $a_1a_2\neq 0$ and the matrix $\widehat{\Phi}$ in \eqref{eq6.12} is nonsingular, then we have
 \begin{equation}\label{eq6.13}
 \text{dist}(\mathcal{X}_1,\mathcal{K}_{2k+2}(\scrH,\bq_1))
 \leq \sqrt{2}\|\Phi^{-1}\|_2\|(X_1^H X_1)^{-\frac{1}{2}}\|_2 \sum_{j=5}^{2n}|a_j|\left| \frac{\l_j}{\l_1}\right|^{k-1},
 \end{equation}
where $\Phi$ is defined in \eqref{eq6.8}, and $\mathcal{X}_1$ and $X_1$ in \eqref{eq6.4}.
\end{theorem}
\begin{proof}
For any given $\bx\in\mathcal{X}_1$ with $\|\bx\|_2=1$, there exists a vector $\bb\in\mathbb{C}^4$ such that $\bx=X_1\bb$.
Then we have
\begin{equation*}
\|\bx\|_2^2=\|X_1\bb\|_2^2=\bb^HX_1^HX_1\bb=\|(X_1^H X_1)^{\frac{1}{2}}\bb\|_2^2=1,
\end{equation*}
leading to
\begin{equation} \label{eq6.14}
\|\bb\|_2=\|(X_1^H X_1)^{-\frac{1}{2}}(X_1^H X_1)^{\frac{1}{2}}\bb\|_2 \leq \|(X_1^H X_1)^{-\frac{1}{2}}\|_2\|(X_1^H X_1)^{\frac{1}{2}}\bb\|_2=\|(X_1^H X_1)^{-\frac{1}{2}}\|_2.
\end{equation}
Since $a_1a_2\neq 0$ and $\widehat{\Phi}$ is nonsingular, we can take the coefficients of
$p_k$ and $\widehat{p}_k$ in \eqref{eq6.5} to be
\begin{equation} \label{eq6.15}
\bz= \begin{bmatrix}
	\zeta_1&\zeta_2&\widehat{\zeta}_1&\widehat{\zeta}_2
\end{bmatrix}^T =\Phi^{-1}\bb.
\end{equation}
Let $\bg(\bx)=p_k(\scrH)\bq_1+\widehat{p}_k(\scrH)\Pi\overline{\bq}_1\in\mathcal{K}_{2k+2}(\scrH,\bq_1)$.
Based on the above analysis, we get
\begin{align*}
&\|\bx-\bg(\bx)\|_2=\|\bx-(p_k(\scrH)\bq_1+\widehat{p}_k(\scrH)\Pi\overline{\bq}_1)\|_2\\
=&\left\|
X_1\bb-X_1\Phi\bz-\sum_{j=5}^{2n} a_j p_k(\l_j)\bx_j 
- \sum_{j=5}^{2n}\overline{a}_j \widehat{p}_k(-\overline{\l}_j)\Pi\overline{\bx}_j
\right\|_2\\
=&\left\| \sum_{j=5}^{2n} a_j p_k(\l_j)\bx_j +
\sum_{j=5}^{2n}\overline{a}_j \widehat{p}_k(-\overline{\l}_j)\Pi\overline{\bx}_j\right\|_2\\
\leq& \sum_{j=5}^{2n}|a_j|
\left[
\left\|\left(\zeta_1\left(\frac{\l_j}{\l_1}\right)^k
 +\zeta_2\left(\frac{\l_j}{\l_1}\right)^{k-1} \right)\bx_j\right\|_2
 +\left\|\left(\widehat{\zeta}_1\left(\frac{\overline{\l}_j}{\l_1}\right)^k 
 +\widehat{\zeta}_2\left(\frac{\overline{\l_j}}{\l_1}\right)^{k-1} \right)
\Pi\overline{\bx}_j\right\|_2
\right]\\
\leq &\sum_{j=5}^{2n}|a_j|
\left(
|\zeta_1|\left|\frac{\l_j}{\l_1}\right|^k
 +|\zeta_2|\left|\frac{\l_j}{\l_1}\right|^{k-1}
+|\widehat{\zeta}_1|\left|\frac{\l_j}{\l_1}\right|^k 
+|\widehat{\zeta}_2|\left|\frac{\l_j}{\l_1}\right|^{k-1}
\right)\\
\leq& \sum_{j=5}^{2n}|a_j|\left|\frac{\l_j}{\l_1}\right|^{k-1}\|\bz\|_1
\leq \sqrt{2}\sum_{j=5}^{2n}|a_j|\left|\frac{\l_j}{\l_1}\right|^{k-1}\|\bz\|_2 
\leq \sqrt{2}\|\Phi^{-1}\|_2\sum_{j=5}^{2n}|a_j|\left|\frac{\l_j}{\l_1}\right|^{k-1}\|\bb\|_2\\
\leq &\sqrt{2}\|\Phi^{-1}\|_2\|(X_1^HX_1)^{-\frac{1}{2}}\|_2\sum_{j=5}^{2n}|a_j|\left|\frac{\l_j}{\l_1}\right|^{k-1}.
\end{align*}
By Definition \ref{def6}, we obtain
\begin{align*}
\text{dist}(\mathcal{X}_1,\mathcal{K}_{2k+2}(\scrH,\bq_1))\leq \max_{
	\substack{
     \sss \bx\in\mathcal{X}_1   \\
       \sss \|\bx\|_2=1
}}\|\bx-\bg(\bx)\|_2\leq \sqrt{2}\|\Phi^{-1}\|_2\|(X_1^HX_1)^{-\frac{1}{2}}\|_2\sum_{j=5}^{2n}|a_j|\left|\frac{\l_j}{\l_1}\right|^{k-1},
\end{align*}
the result to be proved.
\end{proof}

Provided that all eigenvalues of $\scrH$ are semi-simple,
Theorem~\ref{thm6.1} illustrates  the distance between
the $\bPi$-Krylov subspace  $\mathcal{K}_{2k+2}(\scrH,\bq_1)$ and
the invariant subspace $\mathcal{X}_1$ associated with the eigenvalues
$\{\lambda_1, -\lambda, \overline{\lambda}_1, -\overline{\lambda}_1\}$,
whose moduli are the greatest.
The inequality \eqref{eq6.13}
reveals that such distance $\text{dist}(\mathcal{X}_1,\mathcal{K}_{2k+2}(\scrH,\bq_1))$ relies  upon  the gap
between $|\lambda_1|$ and absolute values of the rest eigenvalues,
bearing a strong resemblance to the power method (please refer to
\cite{govl:96}).

\section{Numerical Results}\label{sec:egs}

In this section, we will  solve two Bethe-Salpeter eigenvalue problems to 
test the efficiency of our \GQR\ algorithm and \GLan~algorithm. 
Let the corresponding matrix be $\scrH$.   
To show the relative accuracy of those computed eigenvalues by both methods, 
we use the  shift-and-invert  technique  to refine those approximate eigenvalues, 
and then use them as the exact eigenvalues of $\scrH$. Specifically,
let  $\lambda_j$, $j=1,2,\ldots$  be  the  $j$-th ``exact'' eigenvalue of $\scrH$
and let $(\mu_j, \pmb{z}_j)$, $j=1,2,\ldots$ be the corresponding approximate eigenpairs. 
The relative errors $e(\mu_j)$ of $\mu_j$ is
\begin{equation*}
e(\mu_j)=\frac{|\mu_j-\lambda_j|}{|\lambda_j|}, \qquad j=1,2\ldots.
\end{equation*}
Besides, in the following numerical results  the norm of the residual
$r(\mu_j, \pmb{z}_j)$  of   $(\mu_j, \pmb{z}_j)$  will be given, i.e.,
\begin{align*}
r(\mu_j, \pmb{z}_j)=\frac{\|\scrH\pmb{z}_j-\mu_j\pmb{z}_j\|_1}{(\|\scrH\|_1+|\mu_j|)\|\pmb{z}_j\|_1}, \quad j=1,2\ldots.
\end{align*}

For the sake of fairness we code all programs. All numerical computations 
are carried out by running  MATLAB Version 2016b,
on a Lenovo Pro with a 2.60GHz Intel Core i5-3230M CPU and 4GB RAM, with
machine epsilon  $eps=2.22\times 10^{-16}$.

\begin{example}\label{eg:1}
	
In this example, $A$ and $B$, respectively,  
come from \emph{3Dspectralwave2} and \emph{dielFilterV3clx}:

 \emph{$\text{https://www.cise.ufl.edu/research/sparse/matrices/list\_by\_type.html}$}. 

\noindent We select the last $500$ rows and columns of \emph{3Dspectralwave2} and  
 \emph{dielFilterV3clx}, respectively,  to construct $A$ 
 with $A^H=A\in \mathbb{C}^{500\times 500}$ and $B$ with $B^T=B\in \mathbb{C}^{500\times 500}$, 
 leading to $\scrH\in\mathbb{C}^{1000\times1000}$. Here both A and B are dense matrices. 
 We use our \GQR\ agorithm and the classical  QR method to compute all eigenvalues 
 of $\scrH$.
 
When applying the \GQR\ algorithm to compute approximate eigenpairs $(\mu_j,\bz_j)$,  
two  steps are needed:
\begin{enumerate}
	\item[(S\romannumeral1)] Pursuing the inverse iteration to the 
		{\em $\bPi^{-}$-Hermitian-tridiagonal} matrix $\mathscr{G}$, 
		where $\mathscr{G}$ is obtained by performing a serious 
		$\Gamma$-unitaty transformations 
			to the initial $\scrH$, with $\mu_j^{(0)}$ 
			(the computed eigenvalues of   $\mathscr{G}$) being the 
			shifts to acquire  approximate eigenpais $(\widetilde \mu_j, \widetilde \bu_j)$ 
			of $\mathscr{G}$. Note that with  $(\widetilde \mu_j, \widetilde \bu_j)$ and all  
			$\Gamma$-orthogonal transformations we can obtain the corresponding approximate 
			eigenpais $(\widetilde \mu_j, \widetilde \bz_j)$. 
	\item[(S\romannumeral2)] For $(\widetilde \mu_j, \widetilde \bz_j)$ we  perform  
		inverse iterations on $\scrH$ to get the refined eigenpairs $(\mu_j, \bz_j)$. 
\end{enumerate}
It is worthwhile to point that although we can obtain the approximate eigenpais 
$(\mu_j^{(0)}, \bu_j^{(0)})$ by accumulating all  $\Gamma$-orthogonal transformations, 
such  approximations are fairly poor since $\Gamma$-orthogonal transformations lose 
orthogonality. So we need employ  inverse iterations as stated in 
step (S\romannumeral1), which is very cheap due to the 
special structure of $\mathscr{G}$, 
to refine $(\mu_j^{(0)}, \bu_j^{(0)})$. Besides, step (S\romannumeral2) is 
more expensive then the refinement  in  step (S\romannumeral1)  and one may apply it 
when necessary.

Tables~\ref{tbl:flops} and \ref{tbl:times} respectively   display  
the  flops and the executing times   taken by the \GQR\ algorithm and  the QR method, 
which show that our \GQR\ algorithm takes much less  computation cost 
than the classical QR method. Such  superiority of the  \GQR\ algorithm dues to the 
structure-preserving property, that is, all matrices produced by the  
implicit multishift \GQR\ iteration keep the {\em $\Pi^-$-Hermitian-tridiagonal} structure.

\begin{table}[h!]
\caption{The flop counts of the \GQR\ algorithm and the QR method}
\renewcommand{\arraystretch}{1.1}
\begin{center}
\begin{tabular}{|c|c|p{.6\textwidth}|c|}
\hline
Method & \multicolumn{2}{|c|}{Phase} & Flop \\ \hline
\multirow{2}{*}{\GQR} & i & $\Pi^-$-Hermitian-tridiagonalization on $\scrH$& $48 n^3$ \\ \cline{2-4}
                      & ii & One step of implicit multishift \GQR\ iteration& $720 n$ \\ \hline
\multirow{2}{*}{QR} & i & Hessenberg  reduction by Householder transformations on $\scrH$& $160n^3$ \\ \cline{2-4}
                      & ii & One step of implicit double-shift QR iteration & $480 n^2 $ \\ \hline
\end{tabular}%
\end{center}
\label{tbl:flops}
\end{table}

\begin{table}[h!]
\caption{The executing time of the \GQR\ algorithm and the QR method}
\renewcommand{\arraystretch}{1.1}
\begin{center}
\begin{tabular}{|c|c|p{.6\textwidth}|c|}
\hline
Method & \multicolumn{2}{|c|}{Phase} & Time (secs.) \\ \hline
\multirow{2}{*}{\GQR} & i & $\Pi^-$-Hermitian-tridiagonalization on $\scrH$& $51.8$ \\ \cline{2-4}
                      & ii & Implicit multishift \GQR\ iterations& $4730.4 $ \\ \hline
\multirow{2}{*}{QR} & i & Hessenberg reduction by Householder transformations on $\scrH$& $165.1$ \\ \cline{2-4}
                      & ii & Implicit double-shift QR iterations & $18095.2 $ \\ \hline
\end{tabular}%
\end{center}
\label{tbl:times}
\end{table}

The left figure in Figure~\ref{fig:eg1} plots  the normalized residual norms $r(\wtd\mu_j)$
for $(\tilde{\mu}_j,\wtd Q\tilde{\bz}_j)$ and $r(\mu_j)$ for $(\mu_j,\bz_j)$, 
which respectively are obtained in step (S\romannumeral1) and step (S\romannumeral2) in 
\GQR\ algorithm. While the right one gives  the normalized residual norms $r(\mu_j)$ 
corresponding to the QR method.  The $x$-axis represents the moduli
of  all approximate eigenvalues 
$\mu_j$ satisfying $\Re(\mu_j)\geq 0$ and $\Im(\mu_j)\geq 0$,   
which include  $128$ real, $136$ purely imaginary, and $118$ complex eigenvalues. 
It shows in Figure~\ref{fig:eg1} that   the normalized residual norms 
from step (S\romannumeral1) is decent  and the corresponding eigenpairs could be good 
enough for many applications. If refinement is necessary, then the inverse iterations 
in step  (S\romannumeral2) is  efficient. 

\begin{figure}
{\centering
\begin{tabular}{ccc}
\hspace{-0.3 cm}
\resizebox*{0.47\textwidth}{0.30\textheight}{\includegraphics{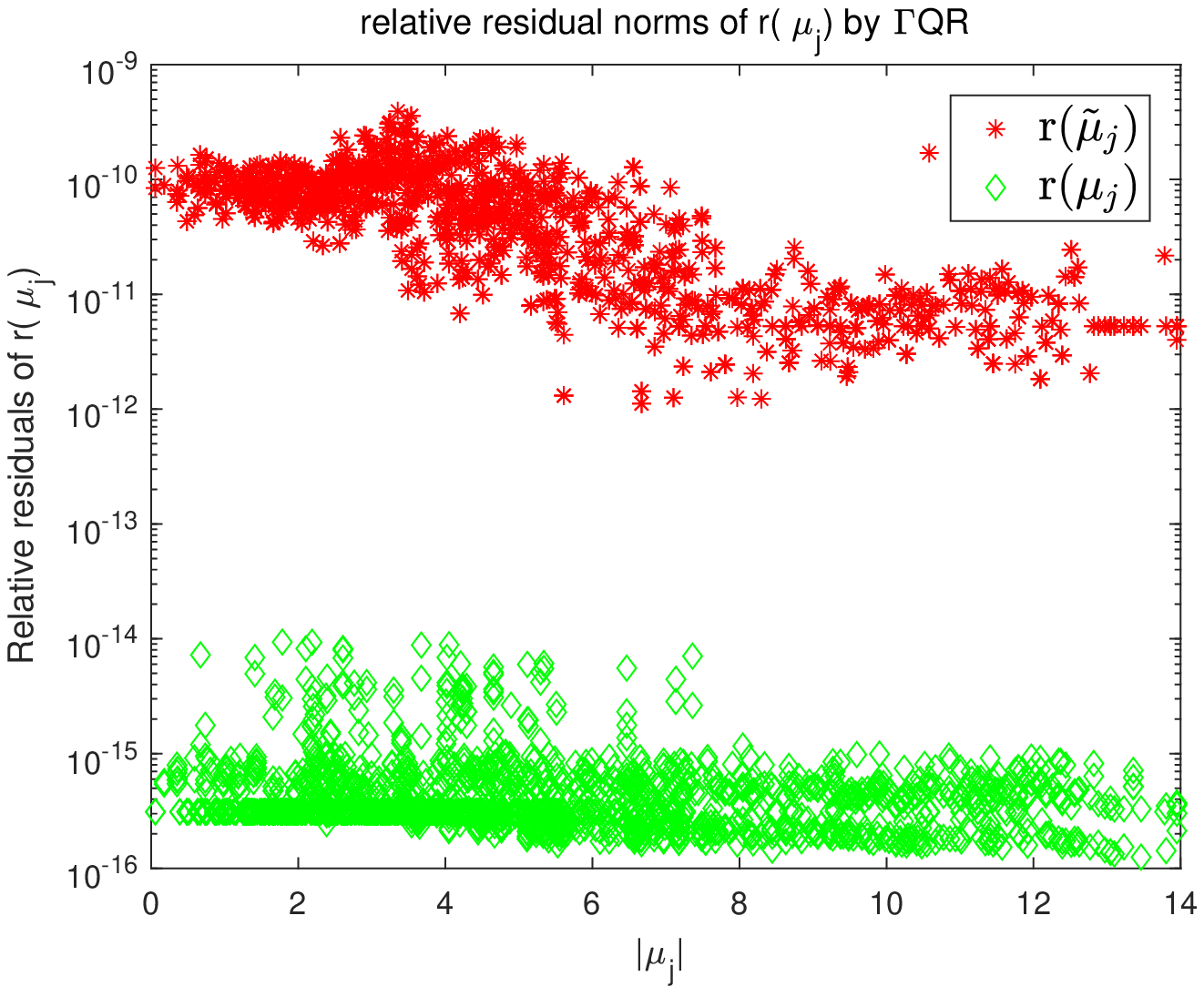}}
& & \hspace{-0.2 cm}
\resizebox*{0.47\textwidth}{0.30\textheight}{\includegraphics{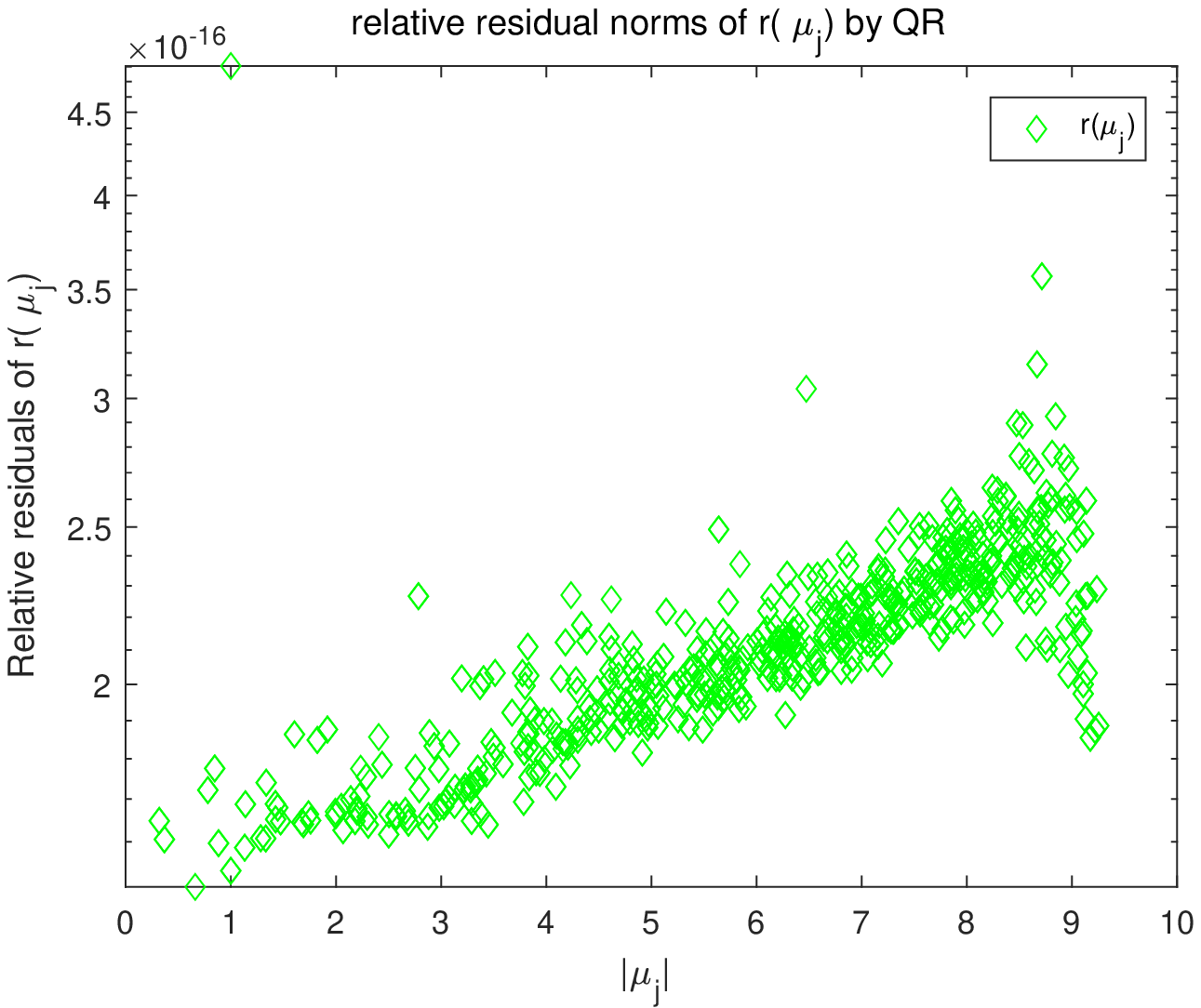}}
\end{tabular}\par
}\vspace{-0.25 cm}
\caption{\small Normalize residual norms for approximate eigenpairs by \GQR\ (left)
   and by QR (right)}
   \label{fig:eg1}
\end{figure}

\end{example}

\begin{example}\label{eg:2}
	
In this example, we borrow the first $10000$ rows and columns of 
\emph{3Dspectralwave2}, which is Hermitian,
to construct $A$,  and the first $10000$  of  \emph{dielFilterV3clx},
which is complex symmetric, to form $B$, 
leading to $\scrH\in\mathbb{C}^{20000\times 20000}$. Here, 
both A and B are sparse matrices. We apply  the  \GLan\ algorithm  and the classical Arnoldi method to compute the largest two eigenvalues of $\scrH^{-1}$ in absolute value, or equivalently, the smallest two eigenvalues of $\scrH$ in absolute value. 
In practice, we solve a linear system to calculate $\scrH^{-1}q$.

Table~\ref{Table_1.1} gives the executing  time for both methods, showing that 
the \GLan~decomposition takes much less time than the Arnoldi decomposition. 
Table~\ref{tab:eigenvalue}  gives the eigenvalues computed by 
our \GLan~algorithm and the Arnoldi method,
where those obtained by the \GLan~algorithm preserve the special structure of the eigenvalues of
the initial matrix $\scrH$, that is, complex eigenvalues appear in quadruples. Besides,
 the Arnoldi method needs to compute all eight eigenvalues, 
 while our \GLan~algorithm just needs compute two eigenvalues for it.

\begin{table}[h!]
\caption{The executing time by the \GLan\ algorithm  and the  Arnoldi method }
\renewcommand{\arraystretch}{1.1}
\begin{center}
\begin{tabular}{|c|c|p{.6\textwidth}|c|}
\hline
Method & \multicolumn{2}{|c|}{Phase} & Time (secs.) \\ \hline
\multirow{2}{*}{\GLan} & i & \GLan ~decomposition of $\scrH$ & 899.9 \\ \cline{2-4}
                      & ii & Implicit \GQR\ algorithm & 0.9648 \\ \hline
\multirow{2}{*}{Arnoldi} & i &  Arnoldi decomposition of $\scrH$  & 1484.8 \\ \cline{2-4}
                      & ii &  QR method   & 1.6170 \\ \hline
\end{tabular}%
\end{center}
\centering%
\label{Table_1.1}
\end{table}

\begin{table}[H]
\footnotesize 
\caption{The minimum eigenvalues in absolute value}\label{tab:eigenvalue}
\begin{tabular}{c||c||c|c|c|c}
\hline
& &\multicolumn{4}{c}{The smallest two eigenvalues  in moduli }
\\
\cline{3-6}
\multirow{2}{*}{\GLan}
&$\mu_1$&0.0064+0.0117i &0.0064-0.0117i &-0.0064+0.0117i &-0.0064-0.0117i  \\
&$\mu_2$&0.0065+0.0202i &0.0065-0.0202i &-0.0065+0.0202i &-0.0065-0.0202i  \\
\cline{3-6}
\hline
\multirow{2}{*}{Arnoldi}
&$\mu_1$&-0.0064-0.0117i & 0.0063-0.0118i & -0.0063+0.0107i & 0.0064+0.0207i \\
&$\mu_2$&0.0066+0.0117i  & -0.0068+0.0203i & 0.0065-0.0203i & -0.0068-0.0026i\\
\hline
\end{tabular}
\end{table}


Figures~\ref{fig:relative-error} and \ref{fig:residual} display the numerical
results of $e(\mu_j)$ and $r(\mu_j, \pmb{z}_j)$, respectively, for $j=1,2$,
where the $x$-axis denotes the steps of iteration (writing as $k$), and the $y$-axes, respectively, denote the
values of  $e(\mu_j)$ and $r(\mu_j, \pmb{z}_j)$. Clearly, $e(\mu_j)$ and $r(\mu_j, \pmb{z}_j)$ from the
\GLan~algorithm decrease much faster than those from the Arnoldi method,
even $r(\mu_j, \pmb{z}_j)$ acquired by our algorithm is square smaller than that by the Arnoldi method.
Besides, it is worthwhile to point that when using $r(\mu_j, \pmb{z}_j)$ as the stop criterion and setting the
tolerance  as $r(\mu_j, \pmb{z}_j)\leq 10^{-12}$, our \GLan\  algorithm and the Arnoldi method, respectively,
require $15$ and $30$ iterations for $\mu_1$, and $17$ and $42$ for $\mu_2$.

\begin{figure}[H]
\centering
\includegraphics[height=10cm,width=12cm]{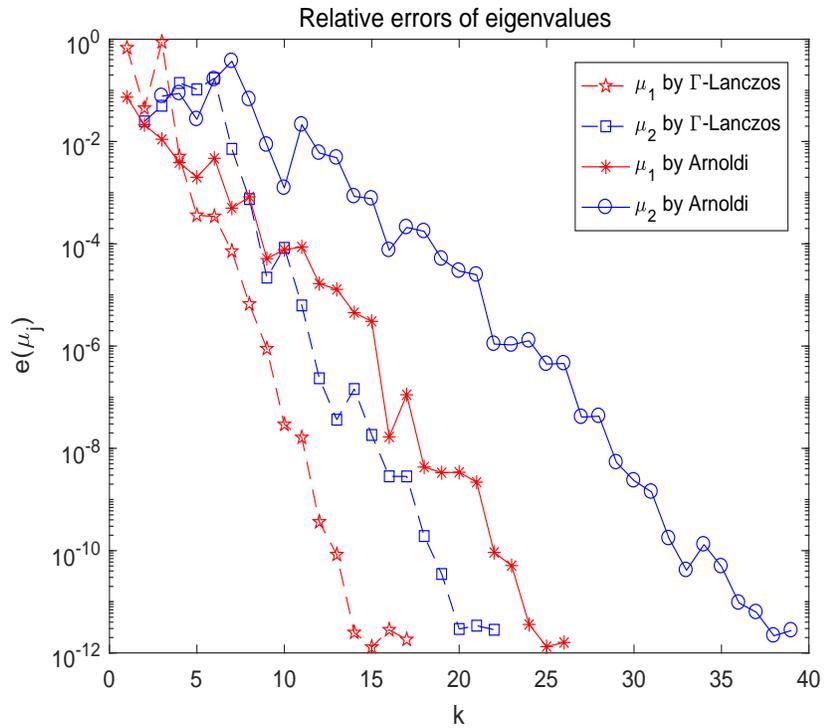}
\caption{Relative error of eigenvalues for \GLan\ and Arnoldi} 
\label{fig:relative-error}
\end{figure}

\begin{figure}[h]
\centering
\includegraphics[height=10cm,width=12cm]{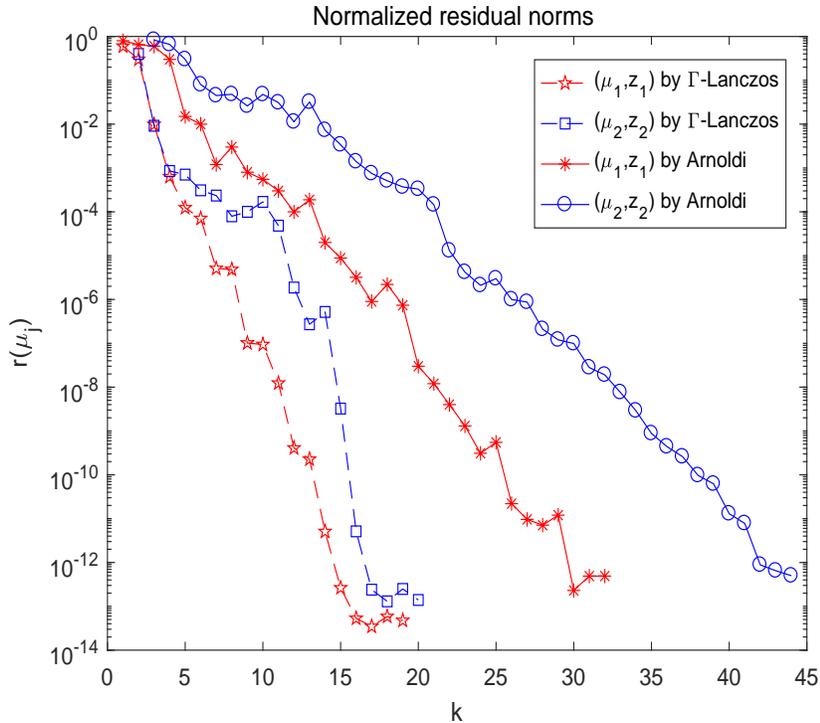}
\caption{Normalized residual norms for \GLan\ and Arnoldi} 
\label{fig:residual}
\end{figure}

\end{example}

\section{Conclusions}\label{sec:concl}

Based on the  \GQR\  algorithm in \cite{lllin:17}   
for solving the linear response eigenvalue problem,
an efficient implicit multishift \GQR\  algorithm is extended to 
solve the Bethe-Salpeter eigenvalue problem
with modest  size, which  preserves the {\em $\Pi^-$-Hermitian}  structure of the initial $\scrH$.
Considering the large-scale Bethe-Salpeter eigenvalue problem, a \GLan\   algorithm is developed,
where the adopted  special projection gives arise a small size   
matrix of the {\em $\Pi^-$-Hermitian} type,
which actually is similar to the Rayleigh Quotient in the  classical Lanczos method.
By computing the eigenvalues of the resulted  small size   matrix,
good approximations of the  eigenpairs of $\scrH$ can  be obtained.
Essentially, the key of both proposed   algorithms is to construct some
$\mathrm{\Gamma}$-unitary transformations which  preserve 
the special structure of the eigenpairs of $\scrH$,
 guaranteeing   the computed eigenpairs appear pairwise as
$\{(\lambda,\bx),(-\overline{\lambda}, \Pi \overline{\bx})\}$. 
Numerical experiments  show that our \GQR~algorithm and  \GLan~algorithm 
respectively take much less executing time  than the QR method and the Arnoldi method,  
to achieve the same relative accuracy of the approximate eigenvalue and
also the same  relative error of the approximate eigenpair.

\section*{Acknowledgment}

T. Li was supported in part by the NSFC grants 11471074.
We would like to thank the National Center of Theoretical Sciences,  and the ST Yau Centre at the National Chiao
Tung University for their support. Also we would like
to express our sincere gratitude to the editor and the
referees for their careful reading and their
valuable comments.

\bibliographystyle{elsarticle-num}
\bibliography{lillRPA}

\end{document}